\documentclass[12pt,reqno]{amsart}
\usepackage{amsmath}
\usepackage{amssymb}
\usepackage{amstext}
\usepackage{a4wide}
\usepackage{graphicx}
\allowdisplaybreaks \numberwithin{equation}{section}
\usepackage{color}
\usepackage{cases}

\numberwithin{equation}{section}

\newtheorem{theorem}{Theorem}[section]
\newtheorem{proposition}[theorem]{Proposition}

\newtheorem{lemma}[theorem]{Lemma}

\theoremstyle{definition}

\newtheorem{definition}[theorem]{Definition}

\theoremstyle{remark}
\newtheorem{remark}[theorem]{Remark}

\begin{document}

\title[Desingularization of 3D steady  Euler equation with helical symmetry]
{Desingularization of 3D steady Euler equations with helical symmetry }

\author{Daomin Cao, Jie Wan}
	
\address{Institute of Applied Mathematics, Chinese Academy of Sciences, Beijing 100190, and University of Chinese Academy of Sciences, Beijing 100049,  P.R. China}
\email{dmcao@amt.ac.cn}
\address{School of Mathematics and Statistics, Beijing Institute of Technology, Beijing 100081,  P.R. China}
\email{wanjie@bit.edu.cn}

%\thanks{This work is partially supported by ARC}

\begin{abstract}
In this paper, we study desingularization of steady solutions of 3D incompressible Euler equation with helical symmetry in a general helical domain. We construct a family of steady Euler flows with helical symmetry, such that the associated vorticities tend asymptotically to a helical vortex filament. The solutions are obtained by solving a  semilinear elliptic problem in divergence form with a parameter. By using the stream-function method, we show  the existence  and asymptotic behavior of  ground state solutions concentrating near a single point as the parameter $ \varepsilon\to 0 $. Qualitative properties of those solutions are also discussed.

\end{abstract}

\maketitle

\section{Introduction and main results}
\subsection{Introduction}
The movement of incompressible  Euler flow  confined in a three-dimensional  domain $ D $ without external force is governed by the following system
\begin{equation}\label{Euler eq}
\begin{cases}
\partial_t\mathbf{v}+(\mathbf{v}\cdot \nabla)\mathbf{v}=-\nabla P,\ \ &D\times (0,T),\\
\nabla\cdot \mathbf{v}=0,\ \ &D\times (0,T),\\
\mathbf{v}\cdot \mathbf{n}=v_n,\ \ &\partial D\times (0,T),\\
\mathbf{v}(x, 0)=\mathbf{v}_0(x), \ \ &D,
\end{cases}
\end{equation}
where $ D\subseteq \mathbb{R}^3 $ is a domain with $ C^\infty $ boundary, $ \mathbf{v}=(v_1,v_2,v_3) $ is the velocity field, $ P$ is the scalar pressure, $ \mathbf{n} $ is the outward unit normal of $ \partial D $ and  $ \mathbf{v}_0 $ is the initial velocity. $v_n$ is a function defined on $\partial D$ satisfying compatibility condition $$\int_{\partial D}v_n d\sigma=0,$$
where $ \sigma $ is the area unit on $ \partial D. $ The third equation of \eqref{Euler eq} means that the net flux of velocity across the boundary is zero. When $ v_n\equiv 0 $, it is the impermeable boundary condition, which means that the normal component of velocity on the boundary is zero.

The  vorticity vector field associated with $\mathbf{v}$ is $ \mathbf{w}=(w_1,w_2,w_3)=curl \mathbf{v}=\nabla\times \mathbf{v} $, which describes the rotation of the fluid.  Then $ \mathbf{w} $ satisfies the vorticity equations
\begin{equation}\label{Euler eq2}
\begin{cases}
\partial_t\mathbf{w}+(\mathbf{v}\cdot \nabla)\mathbf{w}=(\mathbf{w}\cdot \nabla)\mathbf{v},\ \ &D\times (0,T),\\
\nabla\cdot \mathbf{v}=0,\ \ &D\times (0,T),\\
\mathbf{v}\cdot \mathbf{n}=v_n,\ \ &\partial D\times (0,T),\\
\mathbf{w}(x, 0)=\nabla\times \mathbf{v}_0(x), \ \ &D.
\end{cases}
\end{equation}
For background of the 3D incompressible Euler equation, see the classical literature \cite{MB, MP}.

In this paper, we are  devoted to   Euler equations \eqref{Euler eq} with helical symmetry. Let us first define helical symmetric solutions and simplify the vorticity equations \eqref{Euler eq2},  see   \cite{DDMW,Du,ET}. Let $ k>0 $. Define a one-parameter group $ \mathcal{G}_k=\{H_\rho:\mathbb{R}^3\to\mathbb{R}^3 \} $, where
\begin{equation*}
H_\rho(x_1,x_2,x_3)^t=(x_1\cos\rho+x_2\sin\rho, -x_1\sin\rho+x_2\cos\rho, x_3+k\rho)^t.
\end{equation*}
Here $ A^t $ is the transposition of a matrix $ A $. So $ H_\rho $ is a superposition of a rotation in $x_1Ox_2$ plane and a translation in $x_3$ axis. Let
$ R_\rho=\begin{pmatrix}
\cos\rho & \sin\rho & 0 \\
-\sin\rho &\cos\rho & 0 \\
0 & 0 & 1
\end{pmatrix} $
be the rotation with respect to $ x_3 $-axis. Then $ H_\rho(x)=R_\rho(x)+k\rho(0,0,1) $. From a geometric point of view, $ 2k\pi $ corresponds to the pitch of helices.

Define a vector field
\begin{equation*}
\overrightarrow{\zeta}= (x_2, -x_1, k)^t.
\end{equation*}
Then $ \overrightarrow{\zeta}  $ is the field of tangents of symmetry lines of $ \mathcal{G}_k $.

To show what the helical solutions are, we first define  helical domains. A domain $ D\in \mathbb{R}^3 $ is called a $ helical~domain $, if $ H_\rho(D)=D $ for any $\rho$. So $ D $ is invariant under the group $ \mathcal{G}_k $. Let $ \Omega=D\cap\{x\mid x_3=0\} $ be the section of $ D $ over $ x_1Ox_2 $ plane. Then $ D  $ can be generated by $ \Omega $ by  letting $ D=\cup_{\rho\in\mathbb{R}}H_\rho(\Omega) $. Throughout this paper, we always assume that  $ \Omega $ is a simply-connected bounded domain with $ C^\infty $ boundary and $ D $ is a helical domain generated by $ \Omega $.

Now we give the definition of helical functions and vector fields. A scalar function $ h $ is called a $ helical $ function, if 
\begin{equation}\label{helical func}
h(H_{\rho}(x))=h(x)
\end{equation}
for any $ \rho\in\mathbb{R}, x\in D. $ By direct computations it is easy to see that a $ C^1 $ function $ h $ is helical if and only if
\begin{equation*}
\overrightarrow{\zeta}\cdot\nabla h=0.
\end{equation*}

A vector field $ \mathbf{h}=(h_1,h_2,h_3) $ is called a $ helical $ field, if 
\begin{equation}\label{helical vect}
\mathbf{h}(H_{\rho}(x))=R_\rho \mathbf{h}(x)
\end{equation}
for any $ \rho\in\mathbb{R}, x\in D. $ Direct computation shows that a $ C^1 $ vector field $ \mathbf{h} $ is helical if and only if
\begin{equation*}
\overrightarrow{\zeta}\cdot\nabla \mathbf{h}=\mathcal{R}\mathbf{h},
\end{equation*}
where $ \mathcal{R}=\begin{pmatrix}
0 & 1 & 0 \\
-1 &0 & 0 \\
0 & 0 & 0
\end{pmatrix} $ (see \cite{ET}). Helical solutions of \eqref{Euler eq} are then defined as follows.
\begin{definition}
	A function pair ($\mathbf{v}, P$) is called a $ helical$ solution pair of \eqref{Euler eq}, if ($\mathbf{v}, P$) satisfies \eqref{Euler eq} and both vector field $ \mathbf{v} $ and scalar function $ P $ are helical.
\end{definition}
Throughout this paper, helical solutions also need to satisfy  the $ orthogonality~condition $:
\begin{equation}\label{ortho}
\mathbf{v}\cdot \overrightarrow{\zeta}=0,
\end{equation}
that is, the velocity field and $ \overrightarrow{\zeta} $ are orthogonal.

Under the condition \eqref{ortho}, one can check that  the vorticity field $ \mathbf{w} $ satisfies (see \cite{ET})
\begin{equation}\label{w formula}
\mathbf{w}=\frac{w}{k}\overrightarrow{\zeta},
\end{equation}
where $ w:=w_3=\partial_{x_1}v_2-\partial_{x_2}v_1 $, the third component of vorticity field $ \mathbf{w} $, is a helical function. Moreover, the first equation of the vorticity equations \eqref{Euler eq2} is equivalent to
\begin{equation*}
\partial_t \mathbf{w}+(\mathbf{v}\cdot \nabla)\mathbf{w}+\frac{1}{k}w \mathcal{R}\mathbf{v}=0.
\end{equation*}
As a consequence, $ w $ satisfies
\begin{equation}\label{vor eq}
\partial_t w+(\mathbf{v}\cdot \nabla)w=0.
\end{equation}
From \eqref{vor eq} we deduce that, $ w $ satisfies a transport equation, which is very similar to the case of 2D Euler equations (see Yudovich \cite{Y}). Moreover, for a solution $ w $ of \eqref{vor eq}, the vorticity field $ \mathbf{w} $ is determined by \eqref{w formula}.

We now introduce a $ stream~ function $ and reduce the system \eqref{Euler eq2} to a 2D problem. Since $ \mathbf{v} $ is a helical vector field, we have $ \overrightarrow{\zeta}\cdot\nabla \mathbf{v}=\mathcal{R}\mathbf{v} $, which implies that \begin{equation}\label{101}
x_2\partial_{x_1}v_3-x_1\partial_{x_2}v_3+k\partial_{x_3}v_3=0.
\end{equation}
The orthogonal condition shows that
\begin{equation}\label{102}
x_2v_1-x_1v_2+kv_3=0.
\end{equation}
It follows from the incompressible condition, \eqref{101} and \eqref{102} that
\begin{equation*}
\begin{split}
0=&\partial_{x_1}v_1+\partial_{x_2}v_2+\partial_{x_3}v_3=\partial_{x_1}v_1+\partial_{x_2}v_2-\frac{x_2}{k}\partial_{x_1}v_3+\frac{x_1}{k}\partial_{x_2}v_3\\
=&\partial_{x_1}v_1+\partial_{x_2}v_2-\frac{x_2}{k^2}\partial_{x_1}(-x_2v_1+x_1v_2)+\frac{x_1}{k^2}\partial_{x_2}(-x_2v_1+x_1v_2)\\
=&\frac{1}{k^2}\partial_{x_1}[(k^2+x_2^2)v_1-x_1x_2v_2]+\frac{1}{k^2}\partial_{x_2}[(k^2+x_1^2)v_2-x_1x_2v_1].
\end{split}
\end{equation*}
Since $ \Omega $ is simply-connected, we can define a stream function $ \varphi:\Omega\to \mathbb{R} $ such that $ \partial_{x_2}\varphi=\frac{1}{k^2}[(k^2+x_2^2)v_1-x_1x_2v_2],   \partial_{x_1}\varphi=-\frac{1}{k^2}[(k^2+x_1^2)v_2-x_1x_2v_1]$, that is,
\begin{equation*}
\begin{pmatrix}
\partial_{x_1}\varphi  \\
\partial_{x_2}\varphi
\end{pmatrix}=-\frac{1}{k^2}
\begin{pmatrix}
-x_1x_2 & k^2+x_1^2 \\
-(k^2+x_2^2) & x_1x_2
\end{pmatrix}
\begin{pmatrix}
v_1   \\
v_2
\end{pmatrix},
\end{equation*}
or equivalently,
\begin{equation}\label{103}
\begin{pmatrix}
v_1   \\
v_2
\end{pmatrix}
=-\frac{1}{k^2+x_1^2+x_2^2}
\begin{pmatrix}
x_1x_2 & -k^2-x_1^2 \\
k^2+x_2^2 & -x_1x_2
\end{pmatrix}
\begin{pmatrix}
\partial_{x_1}\varphi  \\
\partial_{x_2}\varphi
\end{pmatrix}.
\end{equation}

By the definition of $ w $ and \eqref{103}, we get
\begin{equation}\label{voreq1}
\begin{split}
w=&\partial_{x_1}v_2-\partial_{x_2}v_1=(-\partial_{x_2}, \partial_{x_1})\begin{pmatrix}
v_1   \\
v_2
\end{pmatrix}\\
=&(-\partial_{x_2}, \partial_{x_1})\left( -\frac{1}{k^2+x_1^2+x_2^2}
\begin{pmatrix}
x_1x_2 & -k^2-x_1^2 \\
k^2+x_2^2 & -x_1x_2
\end{pmatrix}
\begin{pmatrix}
\partial_{x_1}\varphi  \\
\partial_{x_2}\varphi
\end{pmatrix}\right)\\
=& -(\partial_{x_1}, \partial_{x_2})\left( \frac{1}{k^2+x_1^2+x_2^2}
\begin{pmatrix}
k^2+x_2^2 & -x_1x_2 \\
-x_1x_2 &  k^2+x_1^2
\end{pmatrix}
\begin{pmatrix}
\partial_{x_1}\varphi  \\
\partial_{x_2}\varphi
\end{pmatrix}\right)\\
=&\mathcal{L}_H\varphi,
\end{split}
\end{equation}
where $ \mathcal{L}_H\varphi=-\text{div}(K_H(x_1,x_2)\nabla\varphi) $ is a second order elliptic operator of divergence type with the coefficient matrix
\begin{equation}\label{coef matrix}
K_H(x_1,x_2)=\frac{1}{k^2+x_1^2+x_2^2}
\begin{pmatrix}
k^2+x_2^2 & -x_1x_2 \\
-x_1x_2 &  k^2+x_1^2
\end{pmatrix}.
\end{equation}
Clearly from the definition of the matrix $ K_H $, $ K_H $ is a positive definite matrix satisfying
\begin{enumerate}
	\item[($\mathcal{K}$1).]  $ K_H $ is smooth, i.e., $(K_{H}(\cdot))_{ij}\in C^{\infty}(\overline{\Omega}) $ for $  i,j=1, 2. $
	\item[($\mathcal{K}$2).] $ \mathcal{L}_H $ is strictly elliptic. Indeed, two eigenvalues of $ K_H $ are $ \lambda_1=1, \lambda_2=\frac{k^2}{k^2+|x|^2} $. So one has,  $$ \frac{k^2}{k^2+|x|^2}|\zeta|^2\le (K_H(x)\zeta|\zeta) \le |\zeta|^2,\ \ \ \ \forall\ x\in \Omega, \ \zeta\in \mathbb{R}^2.$$
\end{enumerate}

From \eqref{vor eq}, \eqref{102} and \eqref{103}, one has
\begin{equation}\label{voreq2}
\begin{split}
0=&\partial_t w+v_1\partial_{x_1}w+v_2\partial_{x_2}w+v_3\partial_{x_3}w\\
=&\partial_t w+v_1\partial_{x_1}w+v_2\partial_{x_2}w+\frac{1}{k}(-x_2v_1+x_1v_2)\cdot\frac{1}{k}(-x_2\partial_{x_1}w+x_1\partial_{x_2}w)\\
=&\partial_t w+\frac{1}{k^2}(v_1,v_2)\begin{pmatrix}
k^2+x_2^2 & -x_1x_2 \\
-x_1x_2 &  k^2+x_1^2
\end{pmatrix}
\begin{pmatrix}
\partial_{x_1}w  \\
\partial_{x_2}w
\end{pmatrix}\\
=& \partial_t w-\frac{1}{k^2(k^2+x_1^2+x_2^2)}
(\partial_{x_1}\varphi, \partial_{x_2}\varphi)
\begin{pmatrix}
x_1x_2 & k^2+x_2^2 \\
-k^2-x_1^2 & -x_1x_2
\end{pmatrix}
\begin{pmatrix}
k^2+x_2^2 & -x_1x_2 \\
-x_1x_2 &  k^2+x_1^2
\end{pmatrix}
\begin{pmatrix}
\partial_{x_1}w  \\
\partial_{x_2}w
\end{pmatrix}\\
=&\partial_t w-(\partial_{x_1}\varphi, \partial_{x_2}\varphi)\begin{pmatrix}
0 & 1 \\
-1 &  0
\end{pmatrix}
\begin{pmatrix}
\partial_{x_1}w  \\
\partial_{x_2}w
\end{pmatrix}\\
=&\partial_t w+\partial_{x_2}\varphi\partial_{x_1}w-\partial_{x_1}\varphi\partial_{x_2}w\\
=&\partial_t w+\nabla w\cdot \nabla^\perp\varphi,
\end{split}
\end{equation}
where $ \perp $ denotes the clockwise rotation through $ \pi/2. $

As for the boundary condition of $ \varphi $, from the fact that $ \mathbf{v} $ is a helical vector field and the domain $ D $ is helical,  $ v_n $ is a helical function defined on $ \partial D $. So $ v_n $ can be generated by $ v_n|_{\partial \Omega} $. Moreover, it follows from $ \mathbf{v}\cdot \mathbf{n}=v_n $ on $ \partial D $ that (see (2.66), \cite{ET})
\begin{equation*}
\begin{split}
v_n|_{\partial \Omega}=&\mathbf{v}\cdot \mathbf{n}|_{\partial \Omega}\\
=&v_1n_1+v_2n_2+(-\frac{x_2}{k}v_1+\frac{x_1}{k}v_2)(-\frac{x_2}{k}n_1+\frac{x_1}{k}n_2)\\
=&\frac{1}{k^2}(v_1,v_2)\begin{pmatrix}
k^2+x_2^2 & -x_1x_2 \\
-x_1x_2 &  k^2+x_1^2
\end{pmatrix}
\begin{pmatrix}
n_1  \\
n_2
\end{pmatrix}.
\end{split}
\end{equation*}
Combining this with \eqref{103}, one has
\begin{equation}\label{voreq3}
\begin{split}
v_n|_{\partial \Omega}=&-\frac{1}{k^2(k^2+x_1^2+x_2^2)}
(\partial_{x_1}\varphi, \partial_{x_2}\varphi)
\begin{pmatrix}
x_1x_2 & k^2+x_2^2 \\
-k^2-x_1^2 & -x_1x_2
\end{pmatrix}
\begin{pmatrix}
k^2+x_2^2 & -x_1x_2 \\
-x_1x_2 &  k^2+x_1^2
\end{pmatrix}
\begin{pmatrix}
n_1  \\
n_2
\end{pmatrix}\\
=&-(\partial_{x_1}\varphi, \partial_{x_2}\varphi)\begin{pmatrix}
0 & 1 \\
-1 &  0
\end{pmatrix}
\begin{pmatrix}
n_1  \\
n_2
\end{pmatrix}\\
=&\nabla^\perp\varphi\cdot \nu.
\end{split}
\end{equation}
Here $ \nu=(n_1,n_2) $ is the two-dimensional vector of the first two component of $ \mathbf{n} $ on $ \partial \Omega $.  Note that $ \nu $ is an outward normal of $ \partial \Omega $.

Thus, the 2D vorticity equations of 3D Euler equations with helical symmetry is as follows
\begin{equation}\label{vor str eq2}
\begin{cases}
\partial_t w+\nabla w\cdot \nabla^\perp\varphi=0,\\
w=\mathcal{L}_H\varphi,\\
\nabla^\perp\varphi\cdot \nu|_{\partial \Omega}=v_n|_{\partial \Omega}.
\end{cases}
\end{equation}
Indeed for a solution pair $ (w,\varphi) $ of the 2D vorticity equation \eqref{vor str eq2}, one can recover the helical velocity field $ \mathbf{v} $ and vorticity field $ \mathbf{w} $ of 3D Euler equations \eqref{Euler eq} by using \eqref{103}, \eqref{102}, \eqref{helical vect}, \eqref{helical func} and \eqref{w formula}.

When considering the case $ v_n\equiv 0 $, which corresponds to the impermeable boundary condition, \eqref{voreq3} implies that $ \varphi$  is a constant  on $ \partial \Omega $. Thus one can choose $ \varphi $ such that $ \varphi\equiv0 $ on $ \partial \Omega. $ The associated vorticity equations then  become
\begin{equation}\label{vor str eq}
\begin{cases}
\partial_t w+\nabla w\cdot \nabla^\perp\varphi=0,\\
w=\mathcal{L}_H\varphi,\\
\varphi|_{\partial \Omega}=0.
\end{cases}
\end{equation}

The research of 3D  Euler equations with helical symmetry has received much attention in recent years. \cite{ET} first proved the global well-posedness of $ L^1\cap L^\infty $ weak solutions of the Euler equation with helical symmetry without vorticity stretching (see \eqref{vor str eq}). Note that this result corresponds to the classical Yudovich's result \cite{Y}, since the structure of vorticity equation \eqref{vor str eq} is similar to that of 2D Euler flows.  \cite{DDMW2} considered travelling-rotating invariant Euler flows with helical symmetry concentrating near a single helical filament in the whole space $ \mathbb{R}^3. $ As for the steady solution of 3D Euler equations with helical symmetry, nonlinear stability for stationary smooth Euler flows with helical symmetry is considered in \cite{Ben} by using the direct method of Lyapunov. More results of the existence and regularity of Euler equation with helical symmetry can be found in \cite{AS, BLN, Du, Jiu} for instance.

In this paper we are interested in vortex desingularization problem of steady 3D Euler equations with helical symmetry. Here the  vortex desingularization problem means that we want to construct a family of ``true'' solutions of Euler equations, such that the corresponding vorticity has a small cross-section and concentrates near a vortex filament. The research of this problem  can be traced back to
Helmholtz \cite{He}, who  first studied the motion of the travelling vortex rings  whose vorticities are supported in toroidal regions with a small  cross-section. Then, many articles  considered the problem. As for the vortex being a tube with a small cross-section whose centerline is a straight line and a circle, which can be reduced to 2D Euler equation and 3D axisymmetric Euler equation respectively, results can be found in \cite{BF80, CLW,CPY,DDMW,DV,FB,SV} and reference therein. For the case of Euler equations with helical symmetry, results seem to be few. D$\acute{\text{a}}$vila et al.  \cite{DDMW2} constructed rotational-invariant Euler flows with helical symmetry in the whole space by considering
\begin{equation*}
\begin{split}
-\text{div}(K_H(x)\nabla u)=f_\varepsilon(u-\alpha|\ln\varepsilon|\frac{|x|^2}{2})\ \ \text{in}\ \ \mathbb{R}^2,
\end{split}
\end{equation*}
where $ f_\varepsilon(t)=\varepsilon^2e^t $. Using the Lyapunov-Schmidt reduction method, the authors proved that solutions will concentrate near a helix in the distributional sense, which satisfies the vortex filament conjecture, see  \cite{JS, JS2}. Desingularization of  rotational-invariant helical solutions in  helical domains with bounded cross section was proved in \cite{CW}. However for the problem of desingularization of steady Euler equations with  helical symmetry, few results give us a positive answer.

Our goal in this article is to solve vortex desingularization problem of steady Euler equations with helical symmetry in general helical domains.
We will construct steady solutions of vorticity equations \eqref{vor str eq2} and \eqref{vor str eq}, such that the associated vorticities have small cross-sections and concentrate near a single point as parameter changes. Accordingly, the vorticity field will  concentrate near a one-dimensional helical filament. Note that both the case of $ v_n\equiv0 $ and that of $ v_n\not\equiv 0 $ are considered. To get these results, we solve solutions of a  semilinear elliptic equation in divergence form (see \eqref{eq1}).  By studying the associated variational structure and using stream function method, we get  the existence and limiting behavior of ground states of these equations.

It should be noted that, Euler equations with helical symmetry can be regarded as the general case of 2D and 3D axisymmetric Euler equations. The cases $ k\to+\infty $ and $ k=0 $ correspond to the 2D Euler equations and 3D axisymmetric Euler equations, respectively. The case  $ k\in (0,+\infty) $ is considered in this paper. In contrast to the 2D and 3D axisymmetric problem, the associated operator $ \mathcal{L}_H $ in vorticity equations \eqref{vor str eq2} is a general  elliptic operator in divergence form, which can bring essential difficulty in studying existence and asymptotic behavior of solutions. It seems impossible to reduce the second-order operator $ \mathcal{L}_H $ to the standard Laplace operator by means of a single change of coordinates. We will give rigorous justification of the relation between the coefficient matrix $ K_H $ and limiting location, energy and the concentration diameter of the ground states, which is totally different from the 2D and 3D axisymmetric cases.

To state our results, we need to introduce some notations first. Let  $ \kappa(w)=\int_{\Omega}w(x)dx$ be the circulation of the vorticity $ w $.  For two sets $ A,B $, we define $ dist(A,B)=\min_{x\in A,y\in B}|x-y| $ the distance between sets $ A$ and $B $ and $ diam(A) $ the diameter of the set $ A $.

We first consider desingularization of steady solutions of 3D Euler equations with helical symmetry with impermeable boundary condition $ v_n\equiv0 $. Since $ (w,\varphi) $ is a steady solution, that is, the distribution of $ w, \varphi $ is independent of $ t $, by \eqref{vor str eq}, $ (w,\varphi) $ satisfies the steady vorticity equations
\begin{equation}\label{steady eq}
\begin{cases}
\nabla w\cdot \nabla^\perp\varphi=0,\\
w=\mathcal{L}_H\varphi,\\
 \varphi|_{\partial \Omega}=0.
\end{cases}
\end{equation}
Formally, if
\begin{equation*}
\mathcal{L}_H\varphi=w=\frac{1}{\varepsilon^2}f(\varphi-\mu), \ \ \ \ \varphi|_{\partial\Omega}=0,
\end{equation*}
for some function $ f $ and constants $ \varepsilon,\mu $, then \eqref{steady eq} automatically holds. To conclude, it suffices to look for solutions of the  semilinear elliptic equations in divergence form
\begin{equation}\label{rot eq2}
\begin{cases}
\mathcal{L}_H\varphi(x)=-\text{div}\cdot(K_H(x)\nabla\varphi(x))=\frac{1}{\varepsilon^2}f(\varphi(x)-\mu),\  &x\in \Omega,\\
\varphi(x)=0,\ &x\in \partial \Omega.
\end{cases}
\end{equation}

Our first result is as follows.
\begin{theorem}\label{thm01}
For every $ k>0, m>0,  0<\varepsilon< 1 $, there exists a family of helical solution pairs $ (\mathbf{v}_\varepsilon, P_\varepsilon)(x,t)\in C^1(D\times \mathbb{R}^+) $ of Euler equations \eqref{Euler eq} such that the support set of $ curl \mathbf{v}_\varepsilon $ is a topological helical tube and the associated vorticity-stream function pair $ (w_\varepsilon,\varphi_\varepsilon) $ is a solution of steady vorticity equations \eqref{steady eq}. Moreover, there holds
\begin{enumerate}
\item $ \mathbf{v}_\varepsilon\cdot \mathbf{n}=0$ on $\partial D. $
\item The support set of $ w_\varepsilon $ is simply-connected and
\begin{equation*}
\lim\limits_{\varepsilon\to 0}\frac{\ln diam(supp(w_\varepsilon))}{\ln\varepsilon}=1.
\end{equation*}
As a consequence, $\lim\limits_{\varepsilon\to 0}diam(supp(w_\varepsilon))=0.$
\item  $ \lim\limits_{\varepsilon\to 0}dist(supp(w_\varepsilon), x^*)=0 $, where $ x^*\in\overline{\Omega} $ satisfies $ |x^*|=\max\limits_{\overline{\Omega}}|x| $.
\item $ \lim\limits_{\varepsilon\to 0}\kappa(w_\varepsilon)=\frac{2 k\pi m}{\sqrt{k^2+|x^*|^2}} $.
\end{enumerate}
\end{theorem}

The solution is constructed by studying the existence and asymptotic behavior of the ground state solutions $ \varphi_\varepsilon $ of equations \eqref{rot eq2} with $ f(t)=t^p_+ $ for  $ p>1 $, $ \mu=m\ln\frac{1}{\varepsilon} $ for some prescribed constant $ m>0 $.
\begin{remark}
In \cite{DDMW2}, D$\acute{\text{a}}$vila et al.  constructed rotational-invariant solutions of vorticity equations  with angular velocity $ \alpha|\ln\varepsilon| $  in $ \mathbb{R}^2 $. However, because of the choice of $ f_\varepsilon $, the support set of vorticity is still the whole plane. Our result shows the existence of a family of steady solutions of \eqref{vor str eq} in a general bounded domain, such that  the corresponding vorticity has non-vanishing circulation with small cross-section and shrinks to a helical filament as $ \varepsilon\to0 $.
\end{remark}

\begin{remark}
By the physical meaning of $ k $, the sign of $ k $ determines two different helical structure. The cases $ k>0 $ and $ k<0 $ correspond  to the left-handed helical structure and right-handed helical structure, respectively. For the case $ k<0 $, one can similarly get solutions of \eqref{rot eq2} concentrating near a single point.
\end{remark}

Our second result is on the desingularization of steady solutions of vorticity equations when the boundary is penetrable. Assume that $ \mathbf{v}\cdot \mathbf{n}=v_n\ln\frac{1}{\varepsilon} $ for some helical function $ v_n\not\equiv 0. $ By \eqref{vor str eq2}, steady solution pairs $ (w,\varphi) $ satisfy
\begin{equation}\label{eq not0}
\begin{cases}
\nabla w\cdot \nabla^\perp\varphi=0,\\
w=\mathcal{L}_H\varphi,\\
\nabla^\perp\varphi\cdot \nu|_{\partial \Omega}=v_n|_{\partial \Omega}\ln\frac{1}{\varepsilon}.
\end{cases}
\end{equation}
Suppose that $q\in C^2(\Omega)\cap C^1(\overline{\Omega}) $ satisfies
\begin{equation}\label{eq not1}
\begin{cases}
\mathcal{L}_H q=0,\\
\nabla^{\perp} q\cdot\nu|_{\partial \Omega}=-v_n|_{\partial \Omega}. \\
\end{cases}
\end{equation}
Note that for a solution $ q $ of \eqref{eq not1}, $ q+C $ is also a solution for any constant $ C $. Thus one can always assume that $ \min_{\overline{\Omega}}q>0 $. Let $ u=\varphi+q\ln\frac{1}{\varepsilon}. $ Then $ u $ satisfies
\begin{equation}\label{steady eq2}
\begin{cases}
\nabla w\cdot \nabla^\perp(u-q\ln\frac{1}{\varepsilon})=0,\\
w=\mathcal{L}_H u,\\
\nabla^\perp u\cdot \nu|_{\partial \Omega}=0.
\end{cases}
\end{equation}
So if
\begin{equation*}
\mathcal{L}_Hu=w=\frac{1}{\varepsilon^2}f(u-q\ln\frac{1}{\varepsilon}), \ \ \ \ u|_{\partial\Omega}=0,
\end{equation*}
for some function $ f $ and constants $ \varepsilon, \mu $, then \eqref{steady eq2} automatically holds. And the solution pairs $ (w,\varphi) $ of \eqref{eq not0} can be obtained by letting $ w=\mathcal{L}_Hu $ and $ \varphi=u-q\ln\frac{1}{\varepsilon} $.

Let $ det(K_H) $ denote the determinant of $ K_H $. Our second result is as follows.
\begin{theorem}\label{thm1}
Let  $ k>0, $ $ q>0 $ satisfy $ \mathcal{L}_Hq=0 $ and $ v_n $ be a helical function defined on $ \partial D $ with $ v_n|_{\partial \Omega}=-\nabla^{\perp} q\cdot\nu|_{\partial \Omega} $. Then for every $  0<\varepsilon< 1 $, there exists a family of helical solution pairs $ (\mathbf{v}_\varepsilon, P_\varepsilon)(x,t)\in C^1(D\times \mathbb{R}^+) $ of Euler equations \eqref{Euler eq} such that the support set of $ curl \mathbf{v}_\varepsilon $ is a topological helical tube and the associated vorticity-stream function pair $ (w_\varepsilon,\varphi_\varepsilon) $ is a  solution of steady vorticity equations \eqref{eq not0}. Moreover, the following conclusions hold
\begin{enumerate}
	\item $ 	\mathbf{v}_\varepsilon\cdot \mathbf{n}=v_n\ln\frac{1}{\varepsilon}$  on  $\partial D. $
	\item The support set of $ w_\varepsilon $ is simply-connected and
	\begin{equation*}
	\lim\limits_{\varepsilon\to 0}\frac{\ln diam(supp(w_\varepsilon))}{\ln\varepsilon}=1.
	\end{equation*}
	\item  $ \lim\limits_{\varepsilon\to 0}dist(supp(w_\varepsilon), x^*)=0 $, where $ x^*\in\overline{\Omega} $   is a minimum point of $ q^2\sqrt{det(K_H)} $, that is, $$  q(x^*)^2\sqrt{det(K_H(x^*))}=\min\limits_{\overline{\Omega} }q^2\sqrt{det(K_H)}. $$
	\item $ \lim\limits_{\varepsilon\to 0}\int_{\Omega}w_\varepsilon(x)dx=2\pi q(x^*)\sqrt{det(K_H(x^*))}. $
	\item Moreover, if $ x^*\in \Omega, $ then there exist  $ R_1,R_2>0 $ satisfying
	\begin{equation*}
	R_1\varepsilon\leq \text{diam}(supp(w_\varepsilon))\leq R_2\varepsilon.
	\end{equation*}

\end{enumerate}

\end{theorem}

We now give some comments of the proof of Theorem \ref{thm1}. The strategy is to consider the existence and limiting behavior of ground state solutions of a semilinear elliptic equation in divergence form, see \eqref{eq1} in section 2. First, using the critical point theory  the existence of mountain pass solutions $ u_\varepsilon $ of \eqref{eq1} with critical value $ c_\varepsilon $ of the corresponding variational functional is proved. Then by choosing proper test functions, we get the upper bound of $ c_\varepsilon $, from which we get the connectness of  the vortex core. Since $ K_H $ is not $ -\Delta $, one can not use the classical test functions to get accurate upper bounded of $ c_\varepsilon $. The boundedness of the energy  of the vortex  core is then obtained. Finally based on the  classical estimates of capacity (see, e.g., \cite{BF80,DV, SV}), we prove the lower bound of $ c_\varepsilon $ and the limiting location  of the core. The key of proof is to show that the concentration point of ground states $ u_\varepsilon $  is a minimum point of $ q^2\sqrt{det(K_H)} $. To this end, the optimal upper and lower bounds  of $ c_\varepsilon $ must be obtained.

\begin{remark}
The results of Theorem \ref{thm1} can be regarded as a general result  of the desingularization of classical planar vortex case (see \cite{LYY,SV}) and the vortex ring case (see  \cite{DV}). Note that the cases of planar vortices and  vortex rings correspond to the coefficient matrix $ K_H(x)=Id $ and $ \frac{1}{r}Id $, respectively.  In \cite{DV}, by considering   solutions of
\begin{equation*}
\begin{cases}
-\text{div}\left (\frac{1}{b}\nabla u\right )=\frac{1}{\varepsilon^2} b\left(u-q\ln\frac{1}{\varepsilon}\right)^{p-1}_+,\ \ &x\in \Omega,\\
u=0,\ \ &x\in\partial \Omega,
\end{cases}
\end{equation*}
where $ b $ is a scalar function and $ q $ is a positive function satisfying $ -\text{div}(\frac{1}{b}\nabla q)=0 $, the authors constructed a family of   $ C^1 $ solutions $ u_\varepsilon $  with nonvanishing circulation  concentrating near a minimizer of $ q^2/b $ as $ \varepsilon\to 0 $. Indeed, if we choose $ K_H(x)=\frac{1}{b}Id $ in Theorem \ref{thm1}, then solutions will concentrate near minimizers of $ q^2\sqrt{det(K_H)}=q^2/b $, which  coincides with the results in \cite{DV}.
\end{remark}

\begin{remark}
Indeed, the existence of solutions of general elliptic  equations in divergence form has been studied by \cite{PS}, who considered a singularly perturbed elliptic problem:
\begin{equation}\label{PS eq}
\begin{split}
-\varepsilon^2 \text{div}(K(x)\nabla u) +V(x)u=u^{p},\ \ \ \ x\in \mathbb{R}^n,
\end{split}
\end{equation}
where $ K(x) $ is strictly positive definite, $ n\geq 3 $, $ p\in (1,\frac{n+2}{n-2}) $ and $ V\in C^{1}(\mathbb{R}^n) $ is positive. The authors constructed solutions concentrating near minimizers of $ V(x)^{\frac{p+1}{p-1}}\sqrt{det(K(x))} $ by the penalization technique. However, it seems that the method in \cite{PS} can not be used in our situation since it depends on the positiveness of $ V $.

\end{remark}

\begin{remark}
Recently, \cite{CW} considered desingularization of rotational-invariant solutions of 3D incompressible Euler equation with helical symmetry in an infinite pipe. Using properties of Green's function of a general uniformly elliptic operator, \cite{CW} proved the existence of mountain pass solutions of Euler equation with helical symmetry, the associated vorticities of which are rotational-invariant and concentrate near a helix. While in this paper, instead of  using properties of  Green's function, we use the estimates of capacity to improve  estimates of the diameter of the vortex core and the energy of ground states in \cite{CW}.
\end{remark}

This paper is organized as follows.  In section 2, we introduce the associated variational structure and prove the existence of mountain pass solutions of \eqref{eq1} for every $ \varepsilon\in (0,1) $. Some fundamental properties which will be used in section 3 are also proved.  In section 3 we prove the   limiting behavior of $ u_\varepsilon $. The proof of  Theorem \ref{thm1} and Theorem \ref{thm01} will be  given in section 4.

\section{Variational problem}

We now consider the following equations
\begin{equation}\label{eq1}
\begin{cases}
\mathcal{L}_Hu=-\text{div}(K_H(x)\nabla u)=\frac{1}{\varepsilon^2} \left( u-q\ln\frac{1}{\varepsilon}\right) ^{p}_+,\ \ &x\in \Omega,\\
u=0,\ \ &x\in\partial \Omega,
\end{cases}
\end{equation}
where  $ p>1 $, $ \Omega\subseteq \mathbb{R}^2 $ is bounded, $ K_H(x)=\frac{1}{k^2+x_1^2+x_2^2}
\begin{pmatrix}
k^2+x_2^2 & -x_1x_2 \\
-x_1x_2 &  k^2+x_1^2
\end{pmatrix} $
and $ q $  is a function defined in $ \overline{\Omega} $ satisfying
\begin{enumerate}
	\item[(Q1).]  $ q \in C^2(\Omega)\cap C^1(\overline{\Omega}) $ and $ q>0 $ in $ \overline{\Omega}. $
	\item[(Q2).] $ q $ is a $ \mathcal{L}_H-harmonic $ function, i.e., $ \mathcal{L}_Hq = -\text{div}(K_H(x)\nabla q )= 0.$
\end{enumerate}

Let $ (K_H(x)\mathbf{a}|\mathbf{b})=\sum_{i,j=1}^2(K_H)_{i,j}(x)a_ib_j $ for two vectors $ \mathbf{a},\mathbf{b} $. Define 
$$ \mathcal{H}(\Omega)=\left\{u\in H^1_0(\Omega)\mid \int_{\Omega}(K_H(x)\nabla u|\nabla u)dx<+\infty \right\} $$
with the norm
\begin{equation*}
||u||_{\mathcal{H}(\Omega)}:=\left(\int_{\Omega}(K_H(x)\nabla u|\nabla u)dx\right) ^\frac{1}{2}.
\end{equation*}
Since $ K_H $ is a positive definite matrix with two positive eigenvalues $ \lambda_1=1$ and $ \lambda_2=\frac{k^2}{k^2+x_1^2+x_2^2} $, two norms $ ||\cdot||_{\mathcal{H}(\Omega)} $ and $ ||\cdot||_{H^1_0(\Omega)} $ are equivalent.

Define the associated energy functional of \eqref{eq1}
\begin{equation}\label{E1}
I_{\varepsilon}(u)=\frac{1}{2}\int_{\Omega}(K_H(x)\nabla u|\nabla u)dx -\frac{1}{(p+1)\varepsilon^2}\int_{\Omega}\left(u-q\ln\frac{1}{\varepsilon}\right)^{p+1}_+dx,\ \ \forall u\in \mathcal{H}(\Omega).
\end{equation}
By the definition of $ \mathcal{H} $,  $ I_{\varepsilon} $ is a well-defined $ C^1 $ functional on $ \mathcal{H} $.

Define the Nehari manifold
\begin{equation}\label{Neh}
\begin{split}
\mathcal{N}_\varepsilon=&\{u\in \mathcal{H}(\Omega)\setminus\{0\}\mid \langle I'_\varepsilon(u),u\rangle =0\}\\
=&\left\{u\in \mathcal{H}(\Omega)\setminus\{0\}\mid \int_{\Omega}(K_H(x)\nabla u|\nabla u)dx=\frac{1}{\varepsilon^2}\int_{\Omega}\left(u-q\ln\frac{1}{\varepsilon}\right)^{p}_+udx\right\}.
\end{split}
\end{equation}

\subsection{Existence of solutions}
First, using the classical critical point theory, we get  ground state solutions of \eqref{eq1}.

Since the nonlinearity $ f(t,x)=\frac{1}{\varepsilon^2}(\left( t-q(x)\ln\frac{1}{\varepsilon}\right) ^p_+ $ for $ p>1 $,  $ I_\varepsilon(u) $ has a mountain pass geometry. Thus we can define the  mountain pass value
\begin{equation*}
 c_\varepsilon=\inf\limits_{\gamma\in \mathcal{P}_\varepsilon}\max\limits_{t\in[0,1]}I_\varepsilon(\gamma(t)),
\end{equation*}
 where
\begin{equation*}
\mathcal{P}_\varepsilon=\{\gamma\in C([0,1], \mathcal{H}(\Omega))\mid \gamma(0)=0, I_\varepsilon(\gamma(1))<0\}.
\end{equation*}
Clearly, $ c_\varepsilon>0. $ We have the following characterization of $ \mathcal{N}_\varepsilon $ and the mountain pass value $ c_\varepsilon $, see \cite{CPY book, JT}.
\begin{lemma}[\cite{CPY book}, Theorem 1.3.7]\label{CPY book}
For any $ u\in \mathcal{N}_\varepsilon $, $ u_+\not\equiv 0. $	For any $ u\in \mathcal{H}(\Omega) $ with $ u_+\not\equiv 0 $, there exists a unique $ t(u)>0 $ such that $ t(u)u\in \mathcal{N}_\varepsilon. $ The value of $ t(u) $ is characterized by the identity
\begin{equation}\label{chara}
I_\varepsilon(t(u)u)=\max \{I_\varepsilon(tu), t>0\}.
\end{equation}
Moreover, there holds
\begin{equation*}
c_\varepsilon\leq\inf\limits_{w\not\equiv 0, w\in \mathcal{H}(\Omega)}\max\limits_{t\geq 0}I_\varepsilon(tw)=\inf\limits_{w\in \mathcal{N}_\varepsilon}I_\varepsilon(w).
\end{equation*}
Finally, if the mountain pass value $ c_\varepsilon $ is a critical value for $ I_\varepsilon $, then $ c_\varepsilon=\inf\limits_{w\in \mathcal{N}_\varepsilon}I_\varepsilon(w) $ is the least nontrivial critical value.
\end{lemma}

Using the mountain pass theorem, we get the existence of mountain pass solutions of $ I_\varepsilon $ with  $ c_\varepsilon $.
\begin{proposition}\label{exist}
$ I_\varepsilon $ has a mountain pass solution with mountain pass value $ c_\varepsilon $. Namely, one can find $ u_\varepsilon\in \mathcal{H}(\Omega) $ satisfying
	\begin{equation}\label{201}
	I_\varepsilon(u_\varepsilon)=c_\varepsilon,\ \ I'_\varepsilon(u_\varepsilon)=0.
	\end{equation}
As a consequence,  there holds $ c_\varepsilon=\inf\limits_{w\not\equiv 0, w\in \mathcal{H}(\Omega)}\max\limits_{t\geq 0}I_\varepsilon(tw)=\inf\limits_{w\in \mathcal{N}_\varepsilon}I_\varepsilon(w). $
\end{proposition}

\begin{proof}

By standard mountain pass theory (see, e.g.,  \S 1.4 in \cite{Wi}), we can prove that there exists  $ u_\varepsilon $ such that $ I'_\varepsilon(u_\varepsilon)=0 $ and $   I_\varepsilon(u_\varepsilon)=c_\varepsilon $. So by Lemma \ref{CPY book}, $  c_\varepsilon=\inf\limits_{w\in \mathcal{N}_\varepsilon}I_\varepsilon(w). $

\iffalse
It remains to prove that  $ c_\varepsilon=\inf\limits_{w\in \mathcal{H}(\Omega)}\max\limits_{t\geq 0}I_\varepsilon(tw)=\inf\limits_{w\in \mathcal{N}_\varepsilon}I_\varepsilon(w). $ On the one hand, since $ u_\varepsilon $ is a critical point of $ I_\varepsilon $, we have $ u_\varepsilon\in \mathcal{N}_\varepsilon $, which implies that
\begin{equation}\label{202}
c_\varepsilon=I_\varepsilon(u_\varepsilon)\geq \inf\limits_{w\in \mathcal{N}_\varepsilon}I_\varepsilon(w).
\end{equation}
On the other hand, since $ \{\gamma(t)=tu\mid  t\geq 0, u\in \mathcal{H}(\Omega)/\{0\},  u_+\neq 0 \}\subseteqq \mathcal{P}_\varepsilon $, we have $ c_\varepsilon\leq \inf\limits_{w\in \mathcal{H}(\Omega), w_+\neq 0}\max\limits_{t\geq 0}I_\varepsilon(tw) $. Moreover, by \eqref{chara}  we get
\begin{equation*}
u\in \mathcal{N}_\varepsilon\ \ \text{if and only if }\ \ I_\varepsilon(u)=\max\limits_{t\geq 0}I_\varepsilon(tu).
\end{equation*}
Thus we have
\begin{equation}\label{203}
 c_\varepsilon\leq\inf\limits_{w\in \mathcal{H}(\Omega), w_+\neq 0}\max\limits_{t\geq 0}I_\varepsilon(tw)=\inf\limits_{w\in \mathcal{N}_\varepsilon}I_\varepsilon(w).
\end{equation}
Note that $ \inf\limits_{w\in \mathcal{H}(\Omega), w_+\neq 0}\max\limits_{t\geq 0}I_\varepsilon(tw)=\inf\limits_{w\in \mathcal{H}(\Omega)}\max\limits_{t\geq 0}I_\varepsilon(tw). $ Combining \eqref{202} and \eqref{203}, we complete the proof.

\fi
\end{proof}

\subsection{Basic properties}
First, we  give some basic properties of $ I_\varepsilon $ and the operator $ \mathcal{L}_H $ as follows.
\begin{lemma}\label{le01}
	For any $ u\in \mathcal{H}(\Omega) $,
	\begin{equation}\label{01}
	\left( \frac{1}{2}-\frac{1}{p+1}\right) ||u||_{\mathcal{H}(\Omega)}^2\leq I_\varepsilon(u)-\frac{1}{p+1}\langle I'_\varepsilon(u),u\rangle.
	\end{equation}
\end{lemma}
\begin{proof}
It follows from the definition of $ I_\varepsilon $ and $ I_\varepsilon' $ that
\begin{equation*}
\begin{split}
&I_\varepsilon(u)-\frac{1}{p+1}\langle I'_\varepsilon(u),u\rangle\\
=&\frac{1}{2}||u||_{\mathcal{H}(\Omega)}^2 -\frac{1}{(p+1)\varepsilon^2}\int_{\Omega}\left(u-q\ln\frac{1}{\varepsilon}\right)^{p+1}_+dx-\left( \frac{1}{p+1}||u||_{\mathcal{H}(\Omega)}^2
-\frac{1}{(p+1)\varepsilon^2}\int_{\Omega}\left(u-q\ln\frac{1}{\varepsilon}\right)^{p}_+udx\right) \\
\geq&\left( \frac{1}{2}-\frac{1}{p+1}\right) ||u||_{\mathcal{H}(\Omega)}^2.
\end{split}
\end{equation*}
\end{proof}

\begin{lemma}\label{le02}
	For any $ u\in \mathcal{H}(\Omega) $,
	\begin{equation}\label{02}
	\int_{\Omega}(K_H(x)\nabla u|\nabla u)dx=\int_{\Omega}q^2\left(K_H(x)\nabla \left( \frac{u}{q}\right) |\nabla \left( \frac{u}{q}\right) \right)dx.
	\end{equation}
\end{lemma}
\begin{proof}
We first claim that
\begin{equation}\label{2-1}
\left(K_H(x)\nabla q| \nabla\left( \frac{u^2}{q}\right) \right)=\left(K_H(x)\nabla u|\nabla u\right)- q^2\left(K_H(x)\nabla \left( \frac{u}{q}\right) |\nabla \left( \frac{u}{q}\right) \right).
\end{equation}
Indeed, we have
\begin{equation*}
(K_{H})_{11}(\partial_1 u)^2-q^2(K_{H})_{11}\left( \partial_1 \left( \frac{u}{q}\right) \right) ^2=(K_{H})_{11}\left( \frac{2u}{q}\partial_1 u\partial_1 q-\frac{u^2}{q^2}(\partial_1 q)^2\right) =(K_{H})_{11}\partial_1 q \partial_1 \left( \frac{u^2}{q}\right),
\end{equation*}
\begin{equation*}
\begin{split}
&((K_{H})_{12}+(K_{H})_{21})\partial_1 u\partial_2 u-q^2((K_{H})_{12}+(K_{H})_{21})\partial_1 \left( \frac{u}{q}\right) \partial_2 \left( \frac{u}{q}\right) \\
=&2(K_{H})_{12}\frac{u}{q}(\partial_1 q\partial_2 u+\partial_2 q\partial_1 u)-2(K_{H})_{12}\frac{u^2}{q^2}\partial_1 q\partial_2 q\\
=&(K_{H})_{12}\partial_1 q \partial_2 \left( \frac{u^2}{q}\right) +(K_{H})_{21}\partial_2 q \partial_1 \left( \frac{u^2}{q}\right),
\end{split}
\end{equation*}
and
\begin{equation*}
(K_{H})_{22}(\partial_2 u)^2-q^2(K_{H})_{22}\left( \partial_2 \left( \frac{u}{q}\right) \right) ^2=(K_{H})_{22}\left( \frac{2u}{q}\partial_2 u\partial_2 q-\frac{u^2}{q^2}(\partial_2 q)^2\right) =(K_{H})_{22}\partial_2 q \partial_2 \left( \frac{u^2}{q}\right).
\end{equation*}
Adding up the above inequalities, we get \eqref{2-1}.

Since $ \mathcal{L}_Hq=0 $ and $ \frac{u^2}{q}\in H^1_0(\Omega) $, we have $ \int_\Omega \left(K_H(x)\nabla q| \nabla\left( \frac{u^2}{q}\right) \right)dx=0 $. Integrating both sides of \eqref{2-1} over $ \Omega $, we get \eqref{02}.

\end{proof}

For any $ \bar{x}=(|\bar{x}|\cos\theta_{\bar{x}}, |\bar{x}|\sin\theta_{\bar{x}})\in \mathbb{R}^2 $, denote $ \bar{R}_{\bar{x}} =\begin{pmatrix}
\cos\theta_{\bar{x}} & \sin\theta_{\bar{x}} \\
-\sin\theta_{\bar{x}} &\cos\theta_{\bar{x}}
\end{pmatrix}$ the rotational transformation matrix through $ \theta_{\bar{x}} $. Then we have
\begin{lemma}\label{le03}
	For any $ \bar{x}\in \mathbb{R}^2 $, there holds
	\begin{equation*}
	K_H(\bar{R}_{\bar{x}} y)=\bar{R}_{\bar{x}}K_H(y)\bar{R}_{\bar{x}}^t,\ \ \ \ \forall y\in \mathbb{R}^2.
	\end{equation*}
	
\end{lemma}
\begin{proof}
Note that for any $ y=(y_1,y_2)^t, $ $ \bar{R}_{\bar{x}} y=(y_1\cos\theta_{\bar{x}}+y_2\sin\theta_{\bar{x}}, y_2\cos\theta_{\bar{x}}-y_1\sin\theta_{\bar{x}})^t $ and $ |\bar{R}_{\bar{x}} y|=|y| $. By the definition of $ K_H $, we get
\begin{equation*}
\begin{split}
&K_H(\bar{R}_{\bar{x}} y)\\
=&\frac{1}{k^2+|y|^2}
\begin{pmatrix}
k^2+(y_2\cos\theta_{\bar{x}}-y_1\sin\theta_{\bar{x}})^2 & -(y_1\cos\theta_{\bar{x}}+y_2\sin\theta_{\bar{x}})(y_2\cos\theta_{\bar{x}}-y_1\sin\theta_{\bar{x}}) \\
-(y_1\cos\theta_{\bar{x}}+y_2\sin\theta_{\bar{x}})(y_2\cos\theta_{\bar{x}}-y_1\sin\theta_{\bar{x}}) &k^2+(y_1\cos\theta_{\bar{x}}+y_2\sin\theta_{\bar{x}})^2
\end{pmatrix}\\
=&\frac{1}{k^2+|y|^2}\begin{pmatrix}
\cos\theta_{\bar{x}} & \sin\theta_{\bar{x}} \\
-\sin\theta_{\bar{x}} &\cos\theta_{\bar{x}}
\end{pmatrix}
\begin{pmatrix}
k^2+y_2^2 & -y_1y_2 \\
-y_1y_2 &k^2+y_1^2
\end{pmatrix}
\begin{pmatrix}
\cos\theta_{\bar{x}} & -\sin\theta_{\bar{x}} \\
\sin\theta_{\bar{x}} &\cos\theta_{\bar{x}}
\end{pmatrix}\\
=&\bar{R}_{\bar{x}}K_H(y)\bar{R}_{\bar{x}}^t.
\end{split}
\end{equation*}
\end{proof}
A direct consequence of Lemma \ref{le03} is the rotational invariance of the problem \eqref{eq1}, which will be used in the proof of Theorem \ref{thm1} in section 4. Define $ \Omega_{\bar{x}}:=\{\bar{R}_{\bar{x}}x\mid x\in \Omega\} $ the  region of $ \Omega $ rotated clockwise by $ \theta_{\bar{x}} $. For a function $ u\in \mathcal{H}(\Omega) $, we define $$ u_{\bar{x}}(x):=u(\bar{R}_{-\bar{x}}x) \ \ \ \ \forall x\in\Omega_{\bar{x}}. $$
So $ u_{\bar{x}}\in \mathcal{H}(\Omega_{\bar{x}}) $. Define $ q_{\bar{x}}(x):=q(\bar{R}_{-\bar{x}}x)$ for any  $x\in\Omega_{\bar{x}}. $ Then we get
\begin{lemma}\label{le04}[Rotational invariance]
	$ u $ is a solution of \eqref{eq1} if and only if $ u_{\bar{x}}\in \mathcal{H}(\Omega_{\bar{x}}) $ satisfies
	\begin{equation}\label{eq1-rot}
	-\text{div}(K_H(x)\nabla u_{\bar{x}})=\frac{1}{\varepsilon^2} \left( u_{\bar{x}}-q_{\bar{x}}\ln\frac{1}{\varepsilon}\right) ^{p}_+,\ \ \ \ \text{in}\ \Omega_{\bar{x}}.
	\end{equation}
	
\end{lemma}
\begin{proof}
For any $ x\in\Omega_{\bar{x}}, $ let $ y=\bar{R}_{-\bar{x}}x $. It is not hard to check that for any function $ g $ and vector field $ \textbf{F}=(f_1,f_2)^t $, there holds
\begin{equation*}
\nabla_x g(\bar{R}_{-\bar{x}}x)=\bar{R}_{\bar{x}} \nabla_y g(y),
\end{equation*}
and
\begin{equation*}
\nabla_x\cdot \textbf{F}(\bar{R}_{-\bar{x}}x)=\nabla_y\cdot(\bar{R}_{-\bar{x}} \textbf{F} ) (y).
\end{equation*}
So by using Lemma \ref{le03}, we get
\begin{equation*}
\begin{split}
-\text{div}(K_H(x)\nabla u_{\bar{x}}(x))=&-\nabla_x\cdot(K_H(x)\nabla_x u(\bar{R}_{-\bar{x}}x))\\
=&-\nabla_x\cdot(K_H(x)(\bar{R}_{\bar{x}}\nabla_y u)(\bar{R}_{-\bar{x}}x))\\
=&-\nabla_x\cdot(\bar{R}_{\bar{x}}K_H(\bar{R}_{-\bar{x}}x)\bar{R}_{\bar{x}}^t(\bar{R}_{\bar{x}}\nabla_y u)(\bar{R}_{-\bar{x}}x))\\
=&-\nabla_y\cdot(\bar{R}_{-\bar{x}}\bar{R}_{\bar{x}}K_H(y)\nabla_y u(y))\\
=&-\nabla_y\cdot(K_H(y)\nabla_y u)(y)\\
=&\frac{1}{\varepsilon^2}\left(u(y)-q(y)\ln\frac{1}{\varepsilon}\right)^p_+\\
=&\frac{1}{\varepsilon^2}\left(u_{\bar{x}}(x)-q_{\bar{x}}(x)\ln\frac{1}{\varepsilon}\right)^{p}_+.
\end{split}
\end{equation*}

\end{proof}

\section{Asymptotic behavior of   $ u_\varepsilon $}

Now, we give the asymptotic behavior of mountain pass solutions $ u_\varepsilon $ of \eqref{eq1}. We first consider the asymptotic behavior of $ u_\varepsilon $ under an extra assumption of $ q $ and $ K_H $:
\begin{enumerate}
	\item[$ (A_1). $] There exist minimum points of $ q^2\sqrt{det(K_H)} $ over $ \overline{\Omega} $ which is on the $ x_1 $-axis.
\end{enumerate}
Indeed, this additional assumption is not essential. Under this assumption, it is convenient to  give an optimal upper bound of $ c_\varepsilon $, see Proposition \ref{le301-upbdd}. The proof of asymptotic behavior of   $ u_\varepsilon $ without the assumption $ (A_1)  $ will be given in section 4.

 Let $ x^*\in \overline{\Omega}\cap\{x=(x_1,x_2)\mid x_2=0\} $ be such that  $$ q^2(x^*)\sqrt{det(K_H(x^*))}=\min\limits_{x\in\overline{\Omega} }q^2(x)\sqrt{det(K_H(x))}. $$ Let $ 0<\varepsilon<1. $
\subsection{Upper bound of $ c_\varepsilon $ }
First, we  compute the upper bound of $ c_\varepsilon $. By choosing proper competitors, we can get the following upper bound of $ c_\varepsilon $.
\begin{proposition}\label{le301-upbdd}
There holds
\begin{equation*}
\limsup_{\varepsilon\to 0}\frac{c_\varepsilon}{\ln\frac{1}{\varepsilon}}\leq \pi q^2\sqrt{det(K_H)}(x^*)=\pi \min\limits_{x\in\overline{\Omega} }q^2\sqrt{det(K_H)}(x).
\end{equation*}
Moreover, if $ x^*\in \Omega $, then
\begin{equation*}
c_\varepsilon\leq \pi \min\limits_{x\in\overline{\Omega} }q^2\sqrt{det(K_H)}(x)\ln\frac{1}{\varepsilon}+O(1).
\end{equation*}
\end{proposition}

\begin{proof}
Let $ U(x) $ be a $ C^\infty $ radially symmetric function such that
\begin{equation*}
\begin{cases}
U(x)\geq 0, \ \ \ \ \ \ &x\in B_1(0),\\
U(x)=\ln\frac{1}{|x|}, \ \ &x\in B_1(0)^c.
\end{cases}
\end{equation*}
For any $ \bar{x}\in \Omega\cap\{x_2=0\} $, we choose $\delta>0$ sufficiently small and a truncation $ \varphi_\delta\in C_c^{\infty}(B_{2\delta}(0)) $ such that
\begin{equation*}
0\leq \phi_\delta\leq 1\ \ \text{in}\ B_{2\delta}(0);\ \ \ \ \phi_\delta= 1\ \ \text{in}\ B_{\delta}(0).
\end{equation*}
For any constants $ l_1,l_2>0 $ (which will be determined later), define
\begin{equation*}
\hat{U}(x_1,x_2)=U\left( \frac{x_1}{l_1}, \frac{x_2}{l_2}\right).
\end{equation*}
So the support set of $ \hat{U}_+ $ is the ellipse $ \left\{(x_1,x_2)\mid \frac{x_1^2}{l_1^2}+\frac{x_2^2}{l_2^2}\leq 1\right\} $. For any set $ A $, define $ \hat{A}=\left\{(x_1,x_2)\mid \left( \frac{x_1}{l_1}, \frac{x_2}{l_2}\right) \in A\right\} $. Let $ \hat{\varphi_\delta}(x_1,x_2)=\varphi_\delta\left( \frac{x_1}{l_1}, \frac{x_2}{l_2}\right). $ Then $ supp(\hat{\varphi_\delta})=\hat{B_{2\delta}} $ and $ \hat{\varphi_\delta}\equiv 1 $ on $ \hat{B_{\delta}} $.

We define for any $ \tau>0 $ a test function
\begin{equation*}
\begin{split}
v^\tau_\varepsilon=q(x)\left( \hat{U}\left( \frac{x-\bar{x}}{\varepsilon}\right) +\ln\frac{\tau}{\varepsilon}\right) \hat{\varphi_\delta}(x-\bar{x})
=\left[q\ln\frac{1}{\varepsilon}+q\left( \hat{U}\left( \frac{x-\bar{x}}{\varepsilon}\right) +\ln\tau\right) \right]\hat{\varphi_\delta}(x-\bar{x}).
\end{split}
\end{equation*}
Then $ v^\tau_\varepsilon\in H^1_0(\Omega) $. Define
\begin{equation*}
\begin{split}
g_\varepsilon(\tau):=\frac{1}{\ln\frac{1}{\varepsilon}}\langle I_\varepsilon'(v^\tau_\varepsilon), v^\tau_\varepsilon\rangle
=\frac{1}{\ln\frac{1}{\varepsilon}}\left(\int_{\Omega}(K_H(x)\nabla v^\tau_\varepsilon| \nabla v^\tau_\varepsilon) dx-\frac{1}{\varepsilon^2}\int_\Omega\left(v^\tau_\varepsilon-q\ln\frac{1}{\varepsilon}\right)^p_+v^\tau_\varepsilon dx\right).
\end{split}
\end{equation*}

We now prove that there exists $ \tau_\varepsilon>0 $ such that $ g_\varepsilon(\tau_\varepsilon)=0 $, that is, $ v^{\tau_\varepsilon}_\varepsilon\in \mathcal{N}_\varepsilon. $ By Lemma \ref{le02},
\begin{equation*}
\begin{split}
&\int_{\Omega}(K_H(x)\nabla v^\tau_\varepsilon| \nabla v^\tau_\varepsilon) dx=\int_{\Omega}q^2\left( K_H(x)\nabla \left( \frac{v^\tau_\varepsilon}{q}\right) | \nabla \left( \frac{v^\tau_\varepsilon}{q}\right) \right)  dx\\
=&\int_{\hat{B}_{2\delta}(\bar{x})}q^2\left( K_H(x)\nabla \left( \frac{v^\tau_\varepsilon}{q}\right) | \nabla \left( \frac{v^\tau_\varepsilon}{q}\right) \right)  dx\\
=&\left({\int_{\hat{B}_{2\delta}(\bar{x})/\hat{B}_{\delta}(\bar{x})}+\int_{\hat{B}_{\delta}(\bar{x})/\hat{B}_{\varepsilon}(\bar{x})}
+\int_{\hat{B}_{\varepsilon}(\bar{x})}}\right)q^2\left( K_H(x)\nabla \left( \frac{v^\tau_\varepsilon}{q}\right) | \nabla \left( \frac{v^\tau_\varepsilon}{q}\right) \right)  dx\\
=&A_1+A_2+A_3.
\end{split}
\end{equation*}
By the definition of $ v^\tau_\varepsilon $, we have
\begin{equation}\label{3-1}
A_1\leq C\left( 1+|\ln\frac{\tau}{\delta}|\right),
\end{equation}
where $ C $ is independent of $ \tau. $ Since $ \varepsilon $ sufficiently small, we can assume that $ \varepsilon<\delta $, which implies that
\begin{equation}\label{3-2}
\begin{split}
A_3=&\int_{\hat{B}_{\varepsilon}(\bar{x})}q^2\left( K_H(x)\nabla \left( \frac{v^\tau_\varepsilon}{q}\right) | \nabla \left( \frac{v^\tau_\varepsilon}{q}\right) \right)  dx=\int_{\hat{B}_{\varepsilon}(\bar{x})}q^2\left( K_H(x)\nabla \hat{U}\left( \frac{x-\bar{x}}{\varepsilon}\right) | \nabla \hat{U}\left( \frac{x-\bar{x}}{\varepsilon}\right) \right)  dx\\
=&\int_{\hat{B}_1(0)}q^2(\varepsilon y+\bar{x})(K_H(\varepsilon y+\bar{x})\nabla \hat{U}(y)| \nabla \hat{U}(y))dy\\
\rightrightarrows & \int_{\hat{B}_1(0)}q^2(\bar{x})(K_H(\bar{x})\nabla \hat{U}(y)| \nabla \hat{U}(y))dy\,\,\,\,\text{as}\,\, \varepsilon\to 0.
\end{split}
\end{equation}
The convergence is uniformly about $ \tau. $

As for $ A_2 $, we have
\begin{equation*}
\begin{split}
A_2=&\int_{\hat{B}_{\delta}(\bar{x})\setminus\hat{B}_{\varepsilon}(\bar{x})}q^2\left( K_H(x)\nabla \left( \frac{v^\tau_\varepsilon}{q}\right) | \nabla \left( \frac{v^\tau_\varepsilon}{q}\right) \right)  dx\\
=&\int_{\hat{B}_{\delta}(0)\setminus\hat{B}_{\varepsilon}(0)}q^2(x+\bar{x})\left( K_H(x+\bar{x})\nabla \hat{U}\left( \frac{x}{\varepsilon}\right) | \nabla \hat{U}\left( \frac{x}{\varepsilon}\right) \right)  dx.
\end{split}
\end{equation*}
Note that for any vector $ a=(a_1,a_2) $,
\begin{equation*}
(K_H(x)a|a)=(K_{H})_{11}(x)a_1^2+((K_{H})_{12}+(K_{H})_{21})(x)a_1a_2+(K_{H})_{22}(x)a_2^2.
\end{equation*}
Hence we have
\begin{equation*}
\begin{split}
A_2=&\int_{\hat{B}_{\delta}(0)\setminus\hat{B}_{\varepsilon}(0)}q^2(K_{H})_{11}(x+\bar{x})\partial_1 \hat{U}\left( \frac{x}{\varepsilon}\right) \partial_1\hat{U}\left( \frac{x}{\varepsilon}\right)   dx\\
&+\int_{\hat{B}_{\delta}(0)\setminus\hat{B}_{\varepsilon}(0)}q^2((K_{H})_{12}+(K_{H})_{21})(x+\bar{x})\partial_1 \hat{U}\left( \frac{x}{\varepsilon}\right) \partial_2\hat{U}\left( \frac{x}{\varepsilon}\right) dx\\
&+\int_{\hat{B}_{\delta}(0)\setminus\hat{B}_{\varepsilon}(0)}q^2(K_{H})_{22}(x+\bar{x})\partial_2 \hat{U}\left( \frac{x}{\varepsilon}\right) \partial_2\hat{U}\left( \frac{x}{\varepsilon}\right)dx.
\end{split}
\end{equation*}
Note that $ \nabla \hat{U}(x_1,x_2)=\left(-\frac{1}{l_1^2}\frac{x_1}{\frac{x_1^2}{l_1^2}+\frac{x_2^2}{l^2_2}}, -\frac{1}{l_2^2}\frac{x_2}{\frac{x_1^2}{l_1^2}+\frac{x_2^2}{l^2_2}}\right)$ on $ \hat{B}_1(0)^c $. Hence direct calculation yields
\begin{equation}\label{3-3}
\begin{split}
&\int_{\hat{B}_{\delta}(0)\setminus\hat{B}_{\varepsilon}(0)}q^2(K_{H})_{11}(x+\bar{x})\partial_1 \hat{U}\left( \frac{x}{\varepsilon}\right) \partial_1\hat{U}\left( \frac{x}{\varepsilon}\right)  dx\\
=&\int_{\hat{B}_{\frac{\delta}{\varepsilon}}(0)\setminus\hat{B}_{1}(0)}q^2(K_{H})_{11}(\varepsilon x+\bar{x})\partial_1 \hat{U}(x)\partial_1\hat{U}(x) dx\\
=&\frac{1}{l_1^4}\int_{\hat{B}_{\frac{\delta}{\varepsilon}}(0)\setminus\hat{B}_{1}(0)}q^2(K_{H})_{11}(\varepsilon x+\bar{x})\frac{x_1^2}{\left( \frac{x_1^2}{l_1^2}+\frac{x_2^2}{l^2_2}\right) ^2}dx\\
=&\frac{l_2}{l_1}\int_{B_{\frac{\delta}{\varepsilon}}(0)\setminus B_{1}(0)}q^2(K_{H})_{11}(\varepsilon (l_1x_1,l_2x_2)+\bar{x})\frac{x_1^2}{(x_1^2+x_2^2)^2}dx\\
=&\frac{l_2}{l_1}\int_{B_{\frac{\delta}{\varepsilon}}(0)\setminus B_{1}(0)}(q^2(K_{H})_{11}(\bar{x})+\varepsilon\nabla (q^2(K_{H})_{11})(\bar{x})\cdot (l_1x_1,l_2x_2)+O(\varepsilon^2|x|^2))\frac{x_1^2}{(x_1^2+x_2^2)^2}dx\\
=&\frac{l_2}{l_1}\pi q^2(K_{H})_{11}(\bar{x})\ln\frac{\delta}{\varepsilon}+O(1),
\end{split}
\end{equation}
where we used $ K_H, q\in C^2 $ and Taylor expansion.

Similarly, we can get
\begin{equation}\label{3-4}
\int_{\hat{B}_{\delta}(0)\setminus\hat{B}_{\varepsilon}(0)}q^2(K_{H})_{22}(x+\bar{x})\partial_2 \hat{U}\left( \frac{x}{\varepsilon}\right) \partial_2\hat{U}\left( \frac{x}{\varepsilon}\right)  dx=\frac{l_1}{l_2}\pi q^2(K_{H})_{22}(\bar{x})\ln\frac{\delta}{\varepsilon}+O(1).
\end{equation}
Since $ \bar{x} $ is on the $ x_1 $ axis, we get $ (K_{H})_{12}(\bar{x})=(K_{H})_{21}(\bar{x})=0 $, which implies that
\begin{equation}\label{3-5}
\begin{split}
&\int_{\hat{B}_{\delta}(0)\setminus\hat{B}_{\varepsilon}(0)}q^2((K_{H})_{12}+(K_{H})_{21})(x+\bar{x})\partial_1 \hat{U}\left( \frac{x}{\varepsilon}\right) \partial_2\hat{U}\left( \frac{x}{\varepsilon}\right)  dx\\
=&\int_{B_{\frac{\delta}{\varepsilon}}(0)\setminus B_{1}(0)}q^2((K_{H})_{12}+(K_{H})_{21})(\bar{x})\frac{x_1^2}{(x_1^2+x_2^2)^2}dx+O(1)\\
=&O(1).
\end{split}
\end{equation}
So by \eqref{3-3}, \eqref{3-4} and \eqref{3-5} we have
\begin{equation}\label{3-6}
A_2=\frac{l_2}{l_1}\pi q^2(K_{H})_{11}(\bar{x})\ln\frac{1}{\varepsilon}+\frac{l_1}{l_2}\pi q^2(K_{H})_{22}(\bar{x})\ln\frac{1}{\varepsilon}+O(1),
\end{equation}
where $ O(1) $ is some bounded quantity independent of $ \tau. $

Choosing $ \frac{l_2}{l_1}=\sqrt{\frac{(K_{H})_{22}}{(K_{H})_{11}}(\bar{x})} $ and by \eqref{3-1}, \eqref{3-2} and \eqref{3-6}, we get
\begin{equation}\label{3-7}
\begin{split}
\int_{\Omega}(K_H(x)\nabla v^\tau_\varepsilon| \nabla v^\tau_\varepsilon) dx=&2\pi q^2\sqrt{(K_{H})_{11}(K_{H})_{22}}(\bar{x})\ln\frac{1}{\varepsilon}+O(1)+C|\ln\tau|\\
=&2\pi q^2\sqrt{det(K_{H})}(\bar{x})\ln\frac{1}{\varepsilon}+O(1)+C|\ln\tau|,
\end{split}
\end{equation}
from which we deduce that
\begin{equation}\label{3-8}
\lim_{\varepsilon\to 0}\frac{1}{\ln\frac{1}{\varepsilon}}\int_{\Omega}(K_H(x)\nabla v^\tau_\varepsilon| \nabla v^\tau_\varepsilon) dx=2\pi q^2\sqrt{det(K_{H})}(\bar{x})
\end{equation}
uniformly in any compact set of $ \tau>0. $

On the other hand, note that
\begin{equation*}
\frac{1}{\varepsilon^2}\int_\Omega\left( v^\tau_\varepsilon-q\ln\frac{1}{\varepsilon}\right)^p_+v^\tau_\varepsilon dx=\frac{1}{\varepsilon^2}\int_\Omega\left( v^\tau_\varepsilon-q\ln\frac{1}{\varepsilon}\right) ^{p+1}_+ dx+\frac{1}{\varepsilon^2}\int_\Omega\left( v^\tau_\varepsilon-q\ln\frac{1}{\varepsilon}\right) ^p_+q\ln\frac{1}{\varepsilon} dx.
\end{equation*}
By the definition of $ v^\tau_\varepsilon $, for  $ \varepsilon $ sufficiently small and every $ x\in \Omega $, we have
\begin{equation*}
\left( v^\tau_\varepsilon(x)-q(x)\ln\frac{1}{\varepsilon}\right) _+=q(x)\left( \hat{U}\left( \frac{x-\bar{x}}{\varepsilon}\right) +\ln\tau\right)_+.
\end{equation*}
Hence we get
\begin{equation}\label{3-9}
\begin{split}
&\frac{1}{\varepsilon^2}\int_\Omega\left( v^\tau_\varepsilon-q\ln\frac{1}{\varepsilon}\right) ^{p+1}_+ dx=\int_{\hat{B}_\tau(0)}q^{p+1}(\bar{x}+\varepsilon y)(\hat{U}(y)+\ln\tau)^{p+1}_+dy\\
\rightrightarrows&q^{p+1}(\bar{x})\int_{\hat{B}_\tau(0)}(\hat{U}(y)+\ln\tau)^{p+1}_+dy,
\end{split}
\end{equation}
and
\begin{equation}\label{3-10}
\begin{split}
\frac{1}{\varepsilon^2\ln\frac{1}{\varepsilon}}\int_\Omega\left( v^\tau_\varepsilon-q\ln\frac{1}{\varepsilon}\right) ^{p}_+q\ln\frac{1}{\varepsilon} dx\rightrightarrows q^{p+1}(\bar{x})\int_{\hat{B}_\tau(0)}(\hat{U}(y)+\ln\tau)^{p}_+dy.
\end{split}
\end{equation}
The convergences are uniformly in any compact set of $ \tau $. By \eqref{3-9} and \eqref{3-10}, we get
\begin{equation}\label{3-11}
\lim_{\varepsilon\to 0}\frac{1}{\varepsilon^2\ln\frac{1}{\varepsilon}}\int_\Omega\left( v^\tau_\varepsilon-q\ln\frac{1}{\varepsilon}\right) ^p_+v^\tau_\varepsilon dx=q^{p+1}(\bar{x})\int_{\hat{B}_\tau(0)}(\hat{U}(y)+\ln\tau)^{p}_+dy.
\end{equation}

It follows from \eqref{3-8} and \eqref{3-11} that for any $ \tau>0 $, $ \lim_{\varepsilon\to 0}g_\varepsilon(\tau)=g(\tau) $, where $ g $ is defined by $ g(\tau)=2\pi q^2\sqrt{det(K_{H})}(\bar{x})-q^{p+1}(\bar{x})\int_{\hat{B}_\tau(0)}(\hat{U}(y)+\ln\tau)^{p}_+dy $, and the convergence is uniformly in any compact set of $ \tau $. Now it is not hard to prove that there exist two numbers $ \tau_1, \tau_2>0 $ such that $ g(\tau_1)<0<g(\tau_2) $. So for $ \varepsilon $ sufficiently small, we have $ g_\varepsilon(\tau_1)<0<g_\varepsilon(\tau_2) $, from which we deduce that, there exists $ \tau_\varepsilon\in (\tau_1,\tau_2) $ satisfying $ g_\varepsilon(\tau_\varepsilon)=0. $ Then $ v^{\tau_\varepsilon}_\varepsilon\in \mathcal{N}_\varepsilon. $

Now by \eqref{3-8}, \eqref{3-11} and $ \tau_\varepsilon\in (\tau_1,\tau_2) $,  we can compute that
\begin{equation*}
\begin{split}
\lim_{\varepsilon\to 0}\frac{1}{\ln\frac{1}{\varepsilon}}I_\varepsilon(v^{\tau_\varepsilon}_\varepsilon)=&\lim_{\varepsilon\to 0}\frac{1}{2\ln\frac{1}{\varepsilon}}\int_{\Omega}(K_H(x)\nabla v^{\tau_\varepsilon}_\varepsilon| \nabla v^{\tau_\varepsilon}_\varepsilon) dx-\lim_{\varepsilon\to 0}\frac{1}{(p+1)\varepsilon^2\ln\frac{1}{\varepsilon}}\int_\Omega\left( v^{\tau_\varepsilon}_\varepsilon-q\ln\frac{1}{\varepsilon}\right) ^{p+1}_+dx\\
=&\pi q^2\sqrt{det(K_{H})}(\bar{x}).
\end{split}
\end{equation*}
Taking the infimum over $ \bar{x}\in\Omega\cap\{x_2=0\} $ and using the assumption that $ x^* $ is on the $ x_1 $ axis, we get $ \limsup\limits_{\varepsilon\to 0}\frac{c_\varepsilon}{\ln\frac{1}{\varepsilon}}\leq \pi q^2(x^*)\sqrt{det(K_H(x^*))}=\pi\min\limits_{x\in\overline{\Omega} }q^2\sqrt{det(K_H)}(x). $

If $ x^*\in \Omega $, we can improve the above estimate as follows. Indeed, choosing $ \bar{x}=x^* $, and using \eqref{3-7} and \eqref{3-9} again, we can get
\begin{equation*}
\begin{split}
c_\varepsilon\leq I_\varepsilon(v^{\tau_\varepsilon}_\varepsilon)=&\frac{1}{2}\int_{\Omega}(K_H(x)\nabla v^{\tau_\varepsilon}_\varepsilon| \nabla v^{\tau_\varepsilon}_\varepsilon) dx-\frac{1}{(p+1)\varepsilon^2}\int_\Omega\left( v^{\tau_\varepsilon}_\varepsilon-q\ln\frac{1}{\varepsilon}\right) ^{p+1}_+dx\\
=&\pi q^2\sqrt{det(K_{H})}(x^*)\ln\frac{1}{\varepsilon}+O(1),
\end{split}
\end{equation*}
where we used the fact that $ \tau_\varepsilon\in (\tau_1,\tau_2) $ and the limit in \eqref{3-9} is uniform on compact sets of $ \tau. $ The proof is thus complete.

\end{proof}

\subsection{Diameter and connectness of the vortex core}
We now prove the connectness and estimate the diameter of the vortex core.
To this end, we define $ A_\varepsilon $ the vortex core of solution $ u_\varepsilon $, that is,
\begin{equation*}
A_\varepsilon=\left\{x\in \Omega\mid u_\varepsilon(x)>q(x)\ln\frac{1}{\varepsilon}\right\}.
\end{equation*}
Clearly by the classical regularity theory of elliptic equations, $ u_\varepsilon\in C^{2,\alpha}(\Omega) $ for any $ \alpha\in(0,1) $ and $ A_\varepsilon $ is an open subset of $ \Omega. $

Define $ \text{diam}(A_\varepsilon)=\max\limits_{x,y\in supp(A_\varepsilon)}|x-y| $. We can prove the connectness and estimate the diameter of $ A_\varepsilon $ as follows.
\begin{proposition}\label{connect}
For every $ \varepsilon>0 $ sufficiently small, $ A_\varepsilon $ is connected and simply connected. Moreover
\begin{equation*}
\lim_{\varepsilon\to 0}\frac{diam(A_\varepsilon)}{dist(A_\varepsilon, \partial \Omega)}=0.
\end{equation*}
As a consequence, $ diam(A_\varepsilon) $ tends to 0 as $ \varepsilon\to 0. $
\end{proposition}

\begin{proof}

Assume that $ A_\varepsilon $ has two components $ A_1, A_2 $. We denote $ \psi_i=\left( u_\varepsilon-q\ln\frac{1}{\varepsilon}\right) \chi_{A_i}\in H^1_0(\Omega) $. Let $ \eta_0>0 $ to be determined later. Define $ \tilde{w}_\varepsilon(s)=u_\varepsilon+s\psi_1-s\eta_0\psi_2$. Then $\tilde{w}_\varepsilon(s)\in H^1_0(\Omega) $ for $ s\geq 0 $ sufficiently small.
By Proposition \ref{exist}, $ t_0 $ is a maximum point of $ I_\varepsilon(t\tilde{w}_\varepsilon) $ if and only if
\begin{equation*}
	t_0\int_{\Omega}(K_H(x)\nabla \tilde{w}_\varepsilon|\nabla \tilde{w}_\varepsilon)dx=\frac{1}{\varepsilon^2}\int_{\Omega}\left( t_0\tilde{w}_\varepsilon-q\ln\frac{1}{\varepsilon}\right) ^p_+ \tilde{w}_\varepsilon dx.
\end{equation*}
	
Let $ D(s,t)=t\int_{\Omega}(K_H(x)\nabla \tilde{w}_\varepsilon|\nabla \tilde{w}_\varepsilon)dx-\frac{1}{\varepsilon^2}\int_{\Omega}(t \tilde{w}_\varepsilon-q\ln\frac{1}{\varepsilon})^p_+ \tilde{w}_\varepsilon dx. $ Then $ D(0,1)=0. $ By $ q>0 $, we have
	\begin{equation*}
	\begin{split}
	\frac{\partial D(0,t)}{\partial t}\bigg|_{t=1}=&\int_{\Omega}(K_H(x)\nabla u_\varepsilon|\nabla u_\varepsilon)dx-\frac{p}{\varepsilon^2}\int_{\Omega}\left( u_\varepsilon-q\ln\frac{1}{\varepsilon}\right) ^{p-1}_+u_\varepsilon^2dx\\
	 =&\frac{1}{\varepsilon^2}\int_{\Omega}\left( \left( u_\varepsilon-q\ln\frac{1}{\varepsilon}\right) ^p_+u_\varepsilon-p\left( u_\varepsilon-q\ln\frac{1}{\varepsilon}\right) ^{p-1}_+u_\varepsilon^2\right) dx<0.
	\end{split}
	\end{equation*}
Hence by the implicit function theorem, there is a function $ t=t(s) $ in the neighborhood of $s=0$ such that $ t(0)=1$ and $ D(s, t(s))=0 $
, which implies that,  $ t(s)\tilde{w}_\varepsilon\in \mathcal{N}_\varepsilon $. Note that
\begin{equation*}
\begin{split}
D_s(0,1)=&2\int_{\Omega}(K_H(x)\nabla u_\varepsilon| \nabla(\psi_1-\eta_0\psi_2))dx-\frac{1}{\varepsilon^2}\int_{\Omega}\left( u_\varepsilon-q\ln\frac{1}{\varepsilon}\right) ^p_+(\psi_1-\eta_0\psi_2)dx\\
&-\frac{1}{\varepsilon^2}\int_{\Omega}p\left( u_\varepsilon-q\ln\frac{1}{\varepsilon}\right) ^{p-1}_+u_\varepsilon(\psi_1-\eta_0\psi_2)dx\\
=&-\frac{1}{\varepsilon^2}\int_{\Omega}\left( u_\varepsilon-q\ln\frac{1}{\varepsilon}\right) ^{p-1}_+\left( (p-1) u_\varepsilon+q\ln\frac{1}{\varepsilon}\right) (\psi_1-\eta_0\psi_2)dx.
\end{split}
\end{equation*}
If we choose
\begin{equation*}
\eta_0=\frac{\int_{\Omega}\left( u_\varepsilon-q\ln\frac{1}{\varepsilon}\right) ^{p-1}_+\left( (p-1) u_\varepsilon+q\ln\frac{1}{\varepsilon}\right) \psi_1dx}{\int_{\Omega}\left( u_\varepsilon-q\ln\frac{1}{\varepsilon}\right) ^{p-1}_+\left( (p-1) u_\varepsilon+q\ln\frac{1}{\varepsilon}\right) \psi_2dx}>0,
\end{equation*}
then by the chain rule,  $ t'(0)=-\frac{D_s(0,1)}{D_t(0,1)}=0 $, which implies that for $s$ small $ t(s)=1+O(s^2) $.	

We calculate $ I_\varepsilon(t(s)\tilde{w}_\varepsilon) $. Since $ supp(\psi_1)\cap supp(\psi_2)=\varnothing $, we obtain
	\begin{equation*}
	\begin{split}
	&\int_{\Omega}(K_H(x)\nabla \tilde{w}_\varepsilon|\nabla \tilde{w}_\varepsilon)dx\\
	=&\int_{\Omega}(K_H(x)\nabla u_\varepsilon|\nabla u_\varepsilon)dx+s^2\int_{\Omega}(K_H(x)\nabla \psi_1|\nabla \psi_1)dx+\eta_0^2s^2\int_{\Omega}(K_H(x)\nabla \psi_2|\nabla \psi_2)dx\\
	&+2s\frac{1}{\varepsilon^2}\int_{\Omega}\left( u_\varepsilon-q\ln\frac{1}{\varepsilon}\right) ^p_+ \psi_1 dx-2s\eta_0\frac{1}{\varepsilon^2}\int_{\Omega}\left( u_\varepsilon-q\ln\frac{1}{\varepsilon}\right) ^p_+ \psi_2 dx.
	\end{split}
	\end{equation*}
Since $ t(s)=1+O(s^2) $, we have
	\begin{equation*}
	\begin{split}
	&\frac{1}{(p+1)\varepsilon^2}\int_{\Omega}\left( t(s)\tilde{w}_\varepsilon-q\ln\frac{1}{\varepsilon}\right) ^{p+1}_+dx\\
	 =&\frac{1}{(p+1)\varepsilon^2}\int_{\Omega}\left( t(s)u_\varepsilon-q\ln\frac{1}{\varepsilon})^{p+1}_+dx+\frac{1}{\varepsilon^2}\int_{\Omega}(t(s)u_\varepsilon-q\ln\frac{1}{\varepsilon}\right) ^p_+(s\psi_1-\eta_0s\psi_2)dx\\
	&+\frac{p}{2\varepsilon^2}\int_{\Omega}\left( t(s)u_\varepsilon-q\ln\frac{1}{\varepsilon}\right) ^{p-1}_+(s\psi_1-\eta_0s\psi_2)^2dx+O(s^{2+\sigma})\\
	 =&\frac{1}{(p+1)\varepsilon^2}\int_{\Omega}\left( t(s)u_\varepsilon-q\ln\frac{1}{\varepsilon}\right) ^{p+1}_+dx+\frac{1}{\varepsilon^2}\int_{\Omega}\left( u_\varepsilon-q\ln\frac{1}{\varepsilon}\right) ^p_+(s\psi_1-\eta_0s\psi_2)dx\\
	&+\frac{p}{2\varepsilon^2}\int_{\Omega}\left( u_\varepsilon-q\ln\frac{1}{\varepsilon}\right) ^{p-1}_+(s^2\psi_1^2+\eta_0^2s^2\psi_2^2)dx+O(s^{2+\sigma}),
	\end{split}
	\end{equation*}
for some $ \sigma>0. $ So by Proposition \ref{exist}, we get
	\begin{equation*}
	\begin{split}
	c_\varepsilon\leq &I_\varepsilon(t(s)\tilde{w}_\varepsilon)\\
	=&\frac{t(s)^2}{2}\int_{\Omega}(K_H(x)\nabla \tilde{w}_\varepsilon|\nabla \tilde{w}_\varepsilon)dx-\frac{1}{(p+1)\varepsilon^2}\int_{\Omega}\left( t(s)\tilde{w}_\varepsilon-q\ln\frac{1}{\varepsilon}\right) ^{p+1}_+dx\\
	=&\frac{t(s)^2}{2}\int_{\Omega}(K_H(x)\nabla u_\varepsilon|\nabla u_\varepsilon)dx-\frac{1}{(p+1)\varepsilon^2}\int_{\Omega}\left( t(s)u_\varepsilon-q\ln\frac{1}{\varepsilon}\right) ^{p+1}_+dx\\
	&+\frac{s^2}{2}\bigg[ \int_{\Omega}(K_H(x)\nabla \psi_1|\nabla \psi_1)dx
	-\frac{p}{\varepsilon^2}\int_{\Omega}\left( u_\varepsilon-q\ln\frac{1}{\varepsilon}\right) ^{p-1}_+\psi_1^2dx\\
	&+\eta_0^2\left(\int_{\Omega}(K_H(x)\nabla \psi_2|\nabla \psi_2)dx-\frac{p}{\varepsilon^2}\int_{\Omega}\left( u_\varepsilon-q\ln\frac{1}{\varepsilon}\right) ^{p-1}_+\psi_2^2dx \right) \bigg]+O(s^{2+\sigma}).
	\end{split}
	\end{equation*}
Note that
	\begin{equation*}
	\begin{split}
\frac{t(s)^2}{2}\int_{\Omega}(K_H(x)\nabla u_\varepsilon|\nabla u_\varepsilon)dx-\frac{1}{(p+1)\varepsilon^2}\int_{\Omega}\left( t(s)u_\varepsilon-q\ln\frac{1}{\varepsilon}\right) ^{p+1}_+dx\leq \max\limits_{t\geq 0}I_\varepsilon(tu_\varepsilon).
	\end{split}
	\end{equation*}
Direct calculation shows that
	\begin{equation*}
	\begin{split}
	&\int_{\Omega}(K_H(x)\nabla \psi_1|\nabla \psi_1)dx-\frac{p}{\varepsilon^2}\int_{\Omega}\left( u_\varepsilon-q\ln\frac{1}{\varepsilon}\right) ^{p-1}_+\psi_1^2dx\\
	=&\int_{A_1}\left( K_H(x)\nabla \left( u_\varepsilon-q\ln\frac{1}{\varepsilon}\right) |\nabla \psi_1\right) dx-\frac{p}{\varepsilon^2}\int_{A_1}\left( u_\varepsilon-q\ln\frac{1}{\varepsilon}\right) ^{p-1}_+\psi_1^2dx\\
	=&-\frac{p-1}{\varepsilon^2}\int_{A_1}(\psi_1)^{p+1}dx< 0.\\
	\end{split}
	\end{equation*}
Here we have used $ \mathcal{L}q=0 $. Similarly,  $ \int_{\Omega}(K_H(x)\nabla \psi_2|\nabla \psi_2)dx-\frac{p}{\varepsilon^2}\int_{\Omega}\left( u_\varepsilon-q\ln\frac{1}{\varepsilon}\right) ^{p-1}_+\psi_2^2dx<0. $
Thus $ c_\varepsilon<\max\limits_{t\geq 0}I_\varepsilon(tu_\varepsilon)=I_\varepsilon(u_\varepsilon)=c_\varepsilon $, which is clearly a contradiction.  So we conclude that $ A_\varepsilon $ is connected.
	
Moreover, we can prove that $ A_\varepsilon $ is simply connected. Let $ U $ be the connected component of $ \Omega\setminus A_\varepsilon $ such that $ \partial \Omega\subseteq \overline{U} $. Note that  $ \Omega\setminus U $ is open and $ \mathcal{L}\left( u_\varepsilon-q\ln\frac{1}{\varepsilon}\right) \geq 0$ in $ \Omega\setminus U $ and $ u_\varepsilon-q\ln\frac{1}{\varepsilon}\geq 0 $ on $ \partial (\Omega\setminus U) $.  By the strong maximum principle, $ u_\varepsilon-q\ln\frac{1}{\varepsilon}> 0 $ in $ \Omega\setminus U. $ Thus $ A_\lambda $ is simply connected.

Finally, by the definition of capacity and $ u_\varepsilon\geq q\ln\frac{1}{\varepsilon} $ on $ A_\varepsilon $, we have
\begin{equation*}
\begin{split}
\left( \ln\frac{1}{\varepsilon}\right) ^2\int_{\Omega\setminus A_\varepsilon}q^2\left( K_H(x)\nabla \left( \frac{u_\varepsilon}{q\ln\frac{1}{\varepsilon}}\right)  | \nabla \left( \frac{u_\varepsilon}{q\ln\frac{1}{\varepsilon}}\right) \right) dx
\geq& \left( \ln\frac{1}{\varepsilon}\right) ^2\inf_{  \Omega}(q^2\lambda_2) \int_{\Omega\setminus A_\varepsilon}\bigg|\nabla \left( \frac{u_\varepsilon}{q\ln\frac{1}{\varepsilon}}\right) \bigg|^2dx\\
\geq& C\left( \ln\frac{1}{\varepsilon}\right) ^2 \lambda_2 \text{cap}(A_\varepsilon, \Omega),
\end{split}
\end{equation*}
where $ \lambda_2(x)=\frac{k^2}{k^2+|x|^2} $ is the smaller eigenvalue of $ K_H. $ Since $ \mathbb{R}^2\setminus\Omega $ is connected and unbounded, by the classical estimates of capacity (see \cite{DV, SV}), we have
\begin{equation*}
\text{cap}(A_\varepsilon, \Omega)\geq \frac{2\pi}{\ln16\left( 1+\frac{2\text{dist}(A_\varepsilon, \partial \Omega)}{\text{diam}(A_\varepsilon)}\right) }.
\end{equation*}
By Lemmas \ref{le01}, \ref{le02} and Proposition \ref{le301-upbdd},
\begin{equation*}
\begin{array}{ll}
\left( \ln\frac{1}{\varepsilon}\right) ^2\int_{\Omega\setminus A_\varepsilon}q^2\left( K_H(x)\nabla \left( \frac{u_\varepsilon}{q\ln\frac{1}{\varepsilon}}\right) | \nabla \left( \frac{u_\varepsilon}{q\ln\frac{1}{\varepsilon}}\right)\right) dx&\leq \int_{\Omega}(K_H(x)\nabla u_\varepsilon| \nabla u_\varepsilon)dx\\
&\leq \frac{2(p+1)}{p-1}I_\varepsilon(u_\varepsilon)\leq c\ln\frac{1}{\varepsilon}.
\end{array}
\end{equation*}
Combining all these inequalities,  we get  $ \lim_{\varepsilon\to 0}\frac{diam(A_\varepsilon)}{dist(A_\varepsilon, \partial \Omega)}=0. $

\end{proof}

Define the energy of the vortex core $ E_c(\varepsilon)=\int_{A_\varepsilon}\left( K_H(x)\nabla\left( u_\varepsilon-q\ln\frac{1}{\varepsilon}\right) | \nabla \left( u_\varepsilon-q\ln\frac{1}{\varepsilon}\right) \right) dx $. We will show that $ E_c(\varepsilon) $ is uniformly bounded with respect to $ \varepsilon. $
\begin{lemma}\label{bdd of core}
	There holds for some $ C $ independent of $ \varepsilon$
	\begin{equation*}
	\int_{A_\varepsilon}\left( K_H(x)\nabla \left( u_\varepsilon-q\ln\frac{1}{\varepsilon}\right)  | \nabla \left( u_\varepsilon-q\ln\frac{1}{\varepsilon}\right) \right) dx\leq C.
	\end{equation*}
	
\end{lemma}
\begin{proof}
Direct calculation yields that
\begin{equation}\label{3-12}
\int_{A_\varepsilon}\left( K_H(x)\nabla \left( u_\varepsilon-q\ln\frac{1}{\varepsilon}\right) | \nabla \left( u_\varepsilon-q\ln\frac{1}{\varepsilon}\right) \right) dx=\frac{1}{\varepsilon^2}\int_{A_\varepsilon}\left( u_\varepsilon-q\ln\frac{1}{\varepsilon}\right) ^{p+1}dx,
\end{equation}
and 
\begin{equation}\label{3-13}
\begin{split}
&\int_{\Omega}(K_H(x)\nabla u_\varepsilon| \nabla u_\varepsilon)dx-\int_{A_\varepsilon}\left( K_H(x)\nabla \left( u_\varepsilon-q\ln\frac{1}{\varepsilon}\right) \big| \nabla \left( u_\varepsilon-q\ln\frac{1}{\varepsilon}\right) \right) dx\\
=&\frac{\ln\frac{1}{\varepsilon}}{\varepsilon^2}\int_{A_\varepsilon}\left( u_\varepsilon-q\ln\frac{1}{\varepsilon}\right) ^{p}qdx.
\end{split}
\end{equation}
By \eqref{3-12}, $(\mathcal{K}2)$ and the classical Gagliardo-Nirenberg inequality, we get
\begin{equation*}
\begin{split}
&\int_{A_\varepsilon}\left( K_H(x)\nabla \left( u_\varepsilon-q\ln\frac{1}{\varepsilon}\right) |\nabla \left( u_\varepsilon-q\ln\frac{1}{\varepsilon}\right) \right) dx\\
\leq &\frac{C^*}{\varepsilon^2}\int_{A_\varepsilon}\left( u_\varepsilon-q\ln\frac{1}{\varepsilon}\right) ^{p}dx\left( \int_{A_\varepsilon}\bigg|\nabla \left( u_\varepsilon-q\ln\frac{1}{\varepsilon}\right) \bigg|^2dx \right)^{\frac{1}{2}} \\
\leq & \frac{C^*}{\sqrt{\inf_{\Omega}\lambda_2}\varepsilon^2}\int_{A_\varepsilon}\left( u_\varepsilon-q\ln\frac{1}{\varepsilon}\right) ^{p}dx\left(\int_{A_\varepsilon}\left( K_H(x)\nabla \left( u_\varepsilon-q\ln\frac{1}{\varepsilon}\right) |\nabla \left( u_\varepsilon-q\ln\frac{1}{\varepsilon}\right) \right) dx \right)^{\frac{1}{2}},
\end{split}
\end{equation*}
which implies that
\begin{equation*}
\begin{split}
\int_{A_\varepsilon}\left( K_H(x)\nabla \left( u_\varepsilon-q\ln\frac{1}{\varepsilon}\right) |\nabla \left( u_\varepsilon-q\ln\frac{1}{\varepsilon}\right) \right) dx\leq C\left(\frac{1}{\varepsilon^2}\int_{A_\varepsilon}\left( u_\varepsilon-q\ln\frac{1}{\varepsilon}\right) ^{p}dx\right)^2.
\end{split}
\end{equation*}
By \eqref{3-13} and Proposition \ref{le301-upbdd},  we get
\begin{equation*}
\frac{1}{\varepsilon^2}\int_{A_\varepsilon}\left (u_\varepsilon-q\ln\frac{1}{\varepsilon}\right )^{p}dx\leq \frac{C}{\ln\frac{1}{\varepsilon}}\int_{\Omega}(K_H(x)\nabla u_\varepsilon| \nabla u_\varepsilon)dx\leq \frac{C}{\ln\frac{1}{\varepsilon}}I_\varepsilon(u_\varepsilon)\leq C.
\end{equation*}
Thus we get
\begin{equation*}
\begin{split}
\int_{A_\varepsilon}\left (K_H(x)\nabla \left (u_\varepsilon-q\ln\frac{1}{\varepsilon}\right )| \nabla \left (u_\varepsilon-q\ln\frac{1}{\varepsilon}\right )\right )dx=\frac{1}{\varepsilon^2}\int_{A_\varepsilon}\left (u_\varepsilon-q\ln\frac{1}{\varepsilon}\right )^{p+1}dx\leq  C.
\end{split}
\end{equation*}

\end{proof}

Using Lemma \ref{bdd of core}, we can get the lower bound of the diameter of the vortex core $ A_\varepsilon $ as follows.

\begin{lemma}\label{lowbdd of diam}
	There exists a constant  $ R_1>0 $ independent of $ \varepsilon$ such that 
	\begin{equation*}
	\text{diam}(A_\varepsilon)\geq R_1\varepsilon.
	\end{equation*}
	
\end{lemma}
\begin{proof}
By \eqref{3-12} and the Sobolev inequality, we have
\begin{equation*}
\begin{split}
&\int_{A_\varepsilon}\left (K_H(x)\nabla \left (u_\varepsilon-q\ln\frac{1}{\varepsilon}\right )| \nabla \left (u_\varepsilon-q\ln\frac{1}{\varepsilon}\right )\right )dx=\frac{1}{\varepsilon^2}\int_{A_\varepsilon}\left (u_\varepsilon-q\ln\frac{1}{\varepsilon}\right )^{p+1}dx\\
\leq&\frac{C_*|A_\varepsilon|}{\varepsilon^2}\left(\int_{A_\varepsilon}\bigg|\nabla \left (u_\varepsilon-q\ln\frac{1}{\varepsilon}\right )\bigg|^2dx \right)^{\frac{p+1}{2}}\\
\leq&\frac{C_*|A_\varepsilon|}{(\inf_{\Omega}\lambda_2)^{\frac{p+1}{2}}\varepsilon^2}\left(\int_{A_\varepsilon}\left (K_H(x)\nabla \left (u_\varepsilon-q\ln\frac{1}{\varepsilon}\right )|\nabla \left (u_\varepsilon-q\ln\frac{1}{\varepsilon}\right )\right )dx \right)^{\frac{p+1}{2}},
\end{split}
\end{equation*}
from which we deduce that
\begin{equation*}
\begin{split}
\frac{|A_\varepsilon|}{\varepsilon^2}\geq C\left(\int_{A_\varepsilon}\left (K_H(x)\nabla \left (u_\varepsilon-q\ln\frac{1}{\varepsilon}\right )|\nabla \left (u_\varepsilon-q\ln\frac{1}{\varepsilon}\right )\right )dx \right)^{\frac{-p+1}{2}}.
\end{split}
\end{equation*}
By Lemma \ref{bdd of core}, we conclude that $ |A_\varepsilon|\geq C\varepsilon^2 $. Thus we complete the proof by using the isoperimetric inequality $ |A_\varepsilon|\leq \pi\text{diam}(A_\varepsilon)^2/4 $.

\end{proof}

\subsection{Asymptotic location of $ A_\varepsilon $}

It follows from  Proposition \ref{connect} that $ \lim_{\varepsilon\to 0}A_\varepsilon=0 $, that is, the vortex core of $ u_\varepsilon $ will shrink to a single point $ \hat{x} $ as $ \varepsilon\to0 $. 
We now prove that the limiting location of $ A_\varepsilon $  is a minimum point of $ q^2\sqrt{det(K_H)} $,  by choosing  test functions suitably and using the classical stream-function method.

\begin{proposition}\label{loca}
There holds
\begin{equation*}
\lim_{\varepsilon\to 0}\text{dist}(A_\varepsilon,\hat{x})=0,
\end{equation*}
where $ \hat{x} $ is a minimizer of $ q^2\sqrt{det(K_H)}. $ As a consequence, there holds
\begin{equation}\label{3-14}
\lim_{\varepsilon\to 0}\frac{c_\varepsilon}{\ln\frac{1}{\varepsilon}}=\pi \min\limits_{\overline{\Omega} }q^2\sqrt{det(K_H)}.
\end{equation}
\end{proposition}

\begin{proof}
It follows from \eqref{3-13} that
\begin{equation*}
\begin{split}
&\frac{\ln\frac{1}{\varepsilon}}{\varepsilon^2}\int_{A_\varepsilon}\left (u_\varepsilon-q\ln\frac{1}{\varepsilon}\right )^{p}qdx\\
=&\int_{\Omega}(K_H(x)\nabla u_\varepsilon| \nabla u_\varepsilon)dx-\int_{A_\varepsilon}\left (K_H(x)\nabla \left (u_\varepsilon-q\ln\frac{1}{\varepsilon}\right )| \nabla \left (u_\varepsilon-q\ln\frac{1}{\varepsilon}\right )\right )dx\\
=&2I_\varepsilon(u_\varepsilon)-\frac{p-1}{(p+1)\varepsilon^2}\int_{A_\varepsilon}\left (u_\varepsilon-q\ln\frac{1}{\varepsilon}\right )^{p+1}dx.
\end{split}
\end{equation*}
Hence by Lemma \ref{bdd of core}, we get
\begin{equation}\label{3-15}
\frac{1}{\varepsilon^2}\int_{A_\varepsilon}\left (u_\varepsilon-q\ln\frac{1}{\varepsilon}\right )^{p}qdx=\frac{2c_\varepsilon}{\ln\frac{1}{\varepsilon}}+O\left (\frac{1}{\ln\frac{1}{\varepsilon}}\right )\leq C.
\end{equation}

For any $ 0<\tau<\sigma\leq 1 $, define $ w^{\sigma, \tau}_\varepsilon:=\min\left \{\frac{\left (u_\varepsilon-q\ln\frac{1}{\sigma}\right )_+}{q\ln\frac{\sigma}{\tau}}, 1\right \}\in H^{1}_0(\Omega) $ and $ A_\varepsilon^\sigma:=\left \{x\in\Omega\mid u_\varepsilon(x)>q(x)\ln\frac{1}{\sigma}\right \} $. Then one computes directly that $  w^{\sigma, \tau}_\varepsilon\equiv1 $ on $ A_\varepsilon^\tau $ and $ supp(w^{\sigma, \tau}_\varepsilon)=\overline{A_\varepsilon^\sigma}. $

We claim that for every $ \varepsilon\leq \tau $,
\begin{equation}\label{3-16}
\ln\frac{\sigma}{\tau}\int_{\Omega}q^2(K_H(x)\nabla w^{\sigma, \tau}_\varepsilon|\nabla w^{\sigma, \tau}_\varepsilon)dx=\frac{1}{\varepsilon^2}\int_{\Omega}\left (u_\varepsilon-q\ln\frac{1}{\varepsilon}\right )_+^{p}qdx.
\end{equation}
Indeed, multiplying both sides of \eqref{eq1} by $ \phi=w^{\sigma, \tau}_\varepsilon q\in H^1_0(\Omega) $ and using integration by parts, we get
\begin{equation*}
\int_{\Omega}(K_H(x)\nabla u_\varepsilon| \nabla (w^{\sigma, \tau}_\varepsilon q))dx=\frac{1}{\varepsilon^2}\int_{\Omega}\left (u_\varepsilon-q\ln\frac{1}{\varepsilon}\right )_+^{p}w^{\sigma, \tau}_\varepsilon qdx.
\end{equation*}
Direct computations show that
\begin{equation*}
\frac{1}{\varepsilon^2}\int_{\Omega}\left (u_\varepsilon-q\ln\frac{1}{\varepsilon}\right )_+^{p}w^{\sigma, \tau}_\varepsilon qdx=\frac{1}{\varepsilon^2}\int_{\Omega}\left (u_\varepsilon-q\ln\frac{1}{\varepsilon}\right )_+^{p}qdx
\end{equation*}
and
\begin{equation*}
\begin{split}
&\int_{\Omega}(K_H(x)\nabla u_\varepsilon| \nabla (w^{\sigma, \tau}_\varepsilon q))dx=\int_{\Omega}\left (K_H(x)\nabla \left (u_\varepsilon-q\ln\frac{1}{\sigma}\right )| \nabla (w^{\sigma, \tau}_\varepsilon q)\right )dx\\
=&\ln\frac{\sigma}{\tau}\int_{\Omega}(K_H(x)\nabla (w^{\sigma, \tau}_\varepsilon q)|\nabla (w^{\sigma, \tau}_\varepsilon q)) dx+\int_{\Omega}\left (K_H(x)\nabla \left (u_\varepsilon-q\ln\frac{1}{\sigma}-\ln\frac{\sigma}{\tau}w^{\sigma, \tau}_\varepsilon q\right )| \nabla (w^{\sigma, \tau}_\varepsilon q)\right )dx\\
=&\ln\frac{\sigma}{\tau}\int_{\Omega}(K_H(x)\nabla (w^{\sigma, \tau}_\varepsilon q)|\nabla (w^{\sigma, \tau}_\varepsilon q)) dx+\int_{A_\varepsilon^\tau}\left (K_H(x)\nabla \left (u_\varepsilon-q\ln\frac{1}{\tau}\right )| \nabla q\right )dx\\
=&\ln\frac{\sigma}{\tau}\int_{\Omega}(K_H(x)\nabla (w^{\sigma, \tau}_\varepsilon q)|\nabla (w^{\sigma, \tau}_\varepsilon q)) dx\\
=&\ln\frac{\sigma}{\tau}\int_{\Omega}q^2(K_H(x)\nabla w^{\sigma, \tau}_\varepsilon|\nabla w^{\sigma, \tau}_\varepsilon)dx,
\end{split}
\end{equation*}
where we have used the assumption $ \mathcal{L}_Hq=0 $ and Lemma \ref{le02}. Thus we get \eqref{3-16}.

By the definition of capacity, \eqref{3-15} and \eqref{3-16}, we get
\begin{equation*}
\begin{split}
\text{cap}(A_\varepsilon^\tau, \Omega)\leq \int_{\Omega}|\nabla w^{1, \tau}_\varepsilon|^2dx
\leq \frac{1}{\inf_{\Omega}q^2\lambda_2}\int_{\Omega}q^2(K_H(x)\nabla w^{1, \tau}_\varepsilon|\nabla w^{1, \tau}_\varepsilon)dx
\leq \frac{C}{\ln\frac{1}{\tau}}.
\end{split}
\end{equation*}
Using the capacity estimates in \cite{DV} again, we get
\begin{equation*}
\frac{2\pi}{\ln16\left (1+\frac{2\text{dist}(A_\varepsilon^\tau, \partial \Omega)}{\text{diam}(A_\varepsilon^\tau)}\right )}\leq \text{cap}(A_\varepsilon^\tau, \Omega),
\end{equation*}
from which we deduce that,
\begin{equation}\label{3-17}
\frac{2\pi}{\ln16\left (1+\frac{2\text{dist}(A_\varepsilon^\tau, \partial \Omega)}{\text{diam}(A_\varepsilon^\tau)}\right )}\leq \frac{C}{\ln\frac{1}{\tau}}.
\end{equation}
So there exist constants $ C_1,C_2>0 $ independent of $ \varepsilon,\tau $, such that for any $ 0<\tau<1 $ and $ \varepsilon\leq \tau $,
$$ \text{diam}(A_\varepsilon^\tau)\leq C_1 \tau^{C_2}.$$  

We now claim that for any $\delta>0$, there exist $ \rho>0 $ and $ 0<\varepsilon_0<\rho $, such that for any $ \varepsilon\in(0,\varepsilon_0) $ and $ x,y\in A_\varepsilon^\rho $,
\begin{equation}\label{5001}
q(x)^2\leq q(y)^2(1+\delta),
\end{equation}
and
\begin{equation}\label{5002}
(K_H(x)\zeta|\zeta)\leq (1+\delta)(K_H(y)\zeta|\zeta),\ \ \ \ \forall\ \zeta\in\mathbb{R}^2.
\end{equation} 
Indeed, since $ \inf_{\Omega}q>0 $ and $ q\in C^{2}(\Omega)\cap C^{1}(\overline{\Omega}) $, it is easy to get \eqref{5001}. By $(\mathcal{K}2)$ and  the regularity of $ K_H $, one can also get \eqref{5002}.

Thus taking $ \sigma=\rho, \tau=\varepsilon $ in \eqref{3-16}, we get for any $ x_\varepsilon\in A_\varepsilon\subseteq A_\varepsilon^\rho $
\begin{equation}\label{3-19}
\begin{split}
\frac{1}{\varepsilon^2\ln\frac{\rho}{\varepsilon}}\int_{\Omega}\left (u_\varepsilon-q\ln\frac{1}{\varepsilon}\right )_+^{p}qdx=\int_{\Omega}q^2(K_H(x)\nabla w^{\rho, \varepsilon}_\varepsilon|\nabla w^{\rho, \varepsilon}_\varepsilon)dx\geq \frac{q^2(x_\varepsilon)}{(1+\delta)^2}\int_{\Omega}(K_H(x_\varepsilon)\nabla w^{\rho, \varepsilon}_\varepsilon|\nabla w^{\rho, \varepsilon}_\varepsilon)dx.
\end{split}
\end{equation}
Define a linear transformation matrix $ T_\varepsilon:\mathbb{R}^2\to\mathbb{R}^2 $ satisfying
\begin{equation*}
T_\varepsilon K_H(x_\varepsilon)T_\varepsilon^t=Id.
\end{equation*}
Then $ |det(T_\varepsilon)|=|det(K_H)(x_\varepsilon)|^{-\frac{1}{2}} $. Let $ \Omega'=T_\varepsilon(\Omega), A_\varepsilon'=T_\varepsilon(A_\varepsilon) $. For any $ y=T_\varepsilon(x)\in \Omega' $, define $ \bar{w}^{\rho, \varepsilon}_\varepsilon(y)=w^{\rho, \varepsilon}_\varepsilon(x)=w^{\rho, \varepsilon}_\varepsilon(T_\varepsilon^{-1}(y)) $. Since $ \bar{w}^{\rho, \varepsilon}_\varepsilon\equiv1 $ on $ A_\varepsilon' $ and $ \bar{w}^{\rho, \varepsilon}_\varepsilon\in H^1_0(\Omega') $, using the capacity estimates again we have
\begin{equation}\label{3-20}
\begin{split}
\int_{\Omega}(K_H(x_\varepsilon)\nabla w^{\rho, \varepsilon}_\varepsilon|\nabla w^{\rho, \varepsilon}_\varepsilon)dx=&\int_{\Omega'}(T_\varepsilon K_H(x_\varepsilon)T_\varepsilon^t \nabla \bar{w}^{\rho, \varepsilon}_\varepsilon|\nabla \bar{w}^{\rho, \varepsilon}_\varepsilon)|det(T_\varepsilon^{-1})|dy\\
=&\sqrt{det(K_H)}(x_\varepsilon)\int_{\Omega'}|\nabla \bar{w}^{\rho, \varepsilon}_\varepsilon|^2dy\\
\geq&\sqrt{det(K_H)}(x_\varepsilon)\cdot \text{cap}(A_\varepsilon', \Omega')\\
\geq&\sqrt{det(K_H)}(x_\varepsilon)\cdot \frac{2\pi}{\ln16\left (1+\frac{2\text{dist}(A_\varepsilon', \partial \Omega')}{\text{diam}(A_\varepsilon')}\right )}\\
\geq&\sqrt{det(K_H)}(x_\varepsilon)\cdot \frac{2\pi}{\ln16\left (1+\frac{C_0\text{dist}(A_\varepsilon, \partial \Omega)}{\text{diam}(A_\varepsilon)}\right )},
\end{split}
\end{equation}
for some $ C_0>0 $ independent of $ \varepsilon. $ Thus by \eqref{3-19} and \eqref{3-20}, we get
\begin{equation}\label{3-21}
q^2\sqrt{det(K_H)}(x_\varepsilon)\leq \frac{(1+\delta)^2}{2\pi\ln\frac{\rho}{\varepsilon}}\ln16\left (1+\frac{C_0\text{dist}(A_\varepsilon, \partial \Omega)}{\text{diam}(A_\varepsilon)}\right )\left( \frac{1}{\varepsilon^2}\int_{\Omega}\left (u_\varepsilon-q\ln\frac{1}{\varepsilon}\right )_+^{p}qdx\right).
\end{equation}
Taking the limit superior in both sides of \eqref{3-21}, using \eqref{3-15} and Proposition \ref{le301-upbdd}, we obtain that for any $ \delta>0 $ and $ x_\varepsilon\in A_\varepsilon $,
\begin{equation*}
\limsup_{\varepsilon\to 0}q^2\sqrt{det(K_H)}(x_\varepsilon)\leq\frac{(1+\delta)^2}{2\pi}\limsup_{\varepsilon\to 0}\frac{\ln16\left (1+\frac{C_0\text{dist}(A_\varepsilon, \partial \Omega)}{\text{diam}(A_\varepsilon)}\right )}{\ln\frac{\rho}{\varepsilon}} \cdot2\pi\min\limits_{x\in\overline{\Omega} }q^2\sqrt{det(K_H)}(x).
\end{equation*}
By Lemma \ref{lowbdd of diam} and $ \text{dist}(A_\varepsilon,\partial \Omega)\leq \text{diam}(\Omega) $, we get
\begin{equation*}
\limsup_{\varepsilon\to 0}\frac{\ln16\left (1+\frac{C_0\text{dist}(A_\varepsilon, \partial \Omega)}{\text{diam}(A_\varepsilon)}\right )}{\ln\frac{\rho}{\varepsilon}}\leq \limsup_{\varepsilon\to 0}\frac{\ln16\left (1+\frac{C}{\varepsilon}\right )}{\ln\frac{\rho}{\varepsilon}}=1.
\end{equation*}
Hence we have
\begin{equation*}
\limsup_{\varepsilon\to 0}q^2\sqrt{det(K_H)}(x_\varepsilon)\leq (1+\delta)^2\min\limits_{x\in\overline{\Omega} }q^2\sqrt{det(K_H)}(x).
\end{equation*}
By the arbitrariness of $ \delta>0 $, we conclude that for any $ x_\varepsilon\in A_\varepsilon, $ $ x_\varepsilon $ tends to $ \hat{x} $, where $ \hat{x} $ is a minimizer of $ q^2\sqrt{det(K_H)}. $

Taking the limit inferior in both sides of \eqref{3-21}, using \eqref{3-15}, and by the  the arbitrariness of $ \delta>0 $, we have
\begin{equation*}
\liminf_{\varepsilon\to 0}\frac{c_\varepsilon}{\ln\frac{1}{\varepsilon}}\geq \pi \min\limits_{\overline{\Omega} }q^2\sqrt{det(K_H)}.
\end{equation*}
Combining this with Proposition \ref{le301-upbdd}, we get \eqref{3-14}.
 The proof is thus complete.

\end{proof}

We can then get estimates of the diameter of $ A_\varepsilon $ as follows.
\begin{lemma}\label{diam}
There holds
\begin{equation}\label{3-22}
\lim_{\varepsilon\to 0}\frac{\ln\frac{\text{dist}(A_\varepsilon, \partial \Omega)}{\text{diam}(A_\varepsilon)}}{\ln\frac{1}{\varepsilon}}=1.
\end{equation}
\end{lemma}
\begin{proof}
On the one hand, by Lemma \ref{lowbdd of diam}, we have
\begin{equation*}
\limsup_{\varepsilon\to 0}\frac{\ln\frac{\text{dist}(A_\varepsilon, \partial \Omega)}{\text{diam}(A_\varepsilon)}}{\ln\frac{1}{\varepsilon}}\leq \limsup_{\varepsilon\to 0}\frac{\ln\frac{C}{\varepsilon}}{\ln\frac{1}{\varepsilon}}=1.
\end{equation*}
On the other hand, taking the limit inferior in both sides of \eqref{3-21} and using \eqref{3-15} and Proposition \ref{le301-upbdd}, we get
\begin{equation*}
\begin{split}
&\frac{(1+\delta)^2}{2\pi}\liminf_{\varepsilon\to 0}\frac{\ln16\left (1+\frac{C_0\text{dist}(A_\varepsilon, \partial \Omega)}{\text{diam}(A_\varepsilon)}\right )}{\ln\frac{\rho}{\varepsilon}} \cdot2\pi\min\limits_{x\in\overline{\Omega} }q^2\sqrt{det(K_H)}(x)\\
\geq& \liminf_{\varepsilon\to 0}q^2\sqrt{det(K_H)}(x_\varepsilon)\\
=& \min\limits_{\overline{\Omega} }q^2\sqrt{det(K_H)},
\end{split}
\end{equation*}
which implies that
\begin{equation*}
\frac{1}{(1+\delta)^2}\leq \liminf_{\varepsilon\to 0}\frac{\ln16\left (1+\frac{C_0\text{dist}(A_\varepsilon, \partial \Omega)}{\text{diam}(A_\varepsilon)}\right )}{\ln\frac{\rho}{\varepsilon}}=\liminf_{\varepsilon\to 0}\frac{\ln\frac{\text{dist}(A_\varepsilon, \partial \Omega)}{\text{diam}(A_\varepsilon)}}{\ln\frac{1}{\varepsilon}}.
\end{equation*}
By the arbitrariness of $ \delta>0, $ we have $ \liminf\limits_{\varepsilon\to 0}\frac{\ln\frac{\text{dist}(A_\varepsilon, \partial \Omega)}{\text{diam}(A_\varepsilon)}}{\ln\frac{1}{\varepsilon}}\geq 1. $ The proof is thus complete.

\end{proof}

\begin{remark}
A direct consequence of Lemmas \ref{diam} and \ref{lowbdd of diam} is that for any $\alpha\in (0,1), $ there exists $ C_1,C_2>0 $ such that
\begin{equation*}
C_1\varepsilon\leq \text{diam}(A_\varepsilon) \leq C_2\varepsilon^\alpha.
\end{equation*}
When the limiting location $ \bar{x} $ of $ A_\varepsilon $ is on the boundary of $ \Omega $, such an estimate is optimal. Similar results have been found for 2D Euler equations and 3D axisymmetric equations, see \cite{DV,LYY, SV} for example. However when $ \bar{x} \in \Omega $, we can improve estimates of $ A_\varepsilon $.
\end{remark}

By Proposition \ref{loca}, we show that the limiting location of $ A_\varepsilon $ is $ x^* $, where $ q^2\sqrt{det(K_H)}(x^*)=\min\limits_{\overline{\Omega} }q^2\sqrt{det(K_H)}. $ Note that  $ \kappa(w_\varepsilon)=\frac{1}{\varepsilon^2}\int_{\Omega}\left (u_\varepsilon-q\ln\frac{1}{\varepsilon}\right )^{p}_+dx $ is the circulation of $ w_\varepsilon=\frac{1}{\varepsilon^2}\left (u_\varepsilon-q\ln\frac{1}{\varepsilon}\right )^{p}_+. $ The limit of $ \kappa(w_\varepsilon) $ can be obtained as follows.
\begin{lemma}\label{circulation}
	There holds
	\begin{equation*}
	\lim_{\varepsilon\to 0}\kappa(w_\varepsilon)=2\pi q\sqrt{det(K_H)}(x^*).
	\end{equation*}
\end{lemma}
\begin{proof}

It follows from \eqref{3-13}, \eqref{3-12} and the definition of $ c_\varepsilon $ that
\begin{equation}\label{3-23}
\begin{split}
\frac{\ln\frac{1}{\varepsilon}}{\varepsilon^2}\int_{A_\varepsilon}\left (u_\varepsilon-q\ln\frac{1}{\varepsilon}\right )^{p}qdx=&\int_{\Omega}(K_H\nabla u_\varepsilon| \nabla u_\varepsilon)dx-\int_{A_\varepsilon}\left (K_H\nabla \left (u_\varepsilon-q\ln\frac{1}{\varepsilon}\right )| \nabla \left (u_\varepsilon-q\ln\frac{1}{\varepsilon}\right )\right )dx\\
=&2c_\varepsilon-\frac{p-1}{(p+1)\varepsilon^2}\int_{A_\varepsilon}\left (u_\varepsilon-q\ln\frac{1}{\varepsilon}\right )^{p+1}dx.
\end{split}
\end{equation}
By Lemma \ref{bdd of core}, we have
\begin{equation*}
\frac{1}{\varepsilon^2}\int_{A_\varepsilon}\left (u_\varepsilon-q\ln\frac{1}{\varepsilon}\right )^{p+1}dx\leq C.
\end{equation*}	
So $ \frac{1}{\varepsilon^2}\int_{A_\varepsilon}\left (u_\varepsilon-q\ln\frac{1}{\varepsilon}\right )^{p}qdx=\frac{2c_\varepsilon}{\ln\frac{1}{\varepsilon}}+O\left (\frac{1}{\ln\frac{1}{\varepsilon}}\right )$. By Proposition \ref{loca} and the fact that $ A_\varepsilon\to x^* $, we get
\begin{equation*}
	\lim_{\varepsilon\to 0}\kappa(w_\varepsilon)=q(x^*)^{-1}\lim_{\varepsilon\to 0}\frac{1}{\varepsilon^2}\int_{A_\varepsilon}\left (u_\varepsilon-q\ln\frac{1}{\varepsilon}\right )^{p}qdx=q(x^*)^{-1} \lim_{\varepsilon\to 0}\frac{2c_\varepsilon}{\ln\frac{1}{\varepsilon}}=2\pi q(x^*)\sqrt{det(K_H(x^*))}.
	\end{equation*}
	
\end{proof}

\subsection{Further analysis when the limiting location of $ A_\varepsilon $ is in $  \Omega $}
When  the limiting location of $ A_\varepsilon $ is in  $  \Omega $, we can improve the results in Proposition \ref{loca} and Lemma \ref{diam} by giving more accurate estimates of lower bound of $ c_\varepsilon $ and upper bound of the diameter of $ A_\varepsilon $. Indeed, we have
\begin{proposition}\label{better esti}
If for any $ x_\varepsilon\in A_\varepsilon $, $ \lim_{\varepsilon\to 0}x_\varepsilon=x^*\in \Omega, $ then
\begin{equation}\label{3-27}
c_\varepsilon=I_\varepsilon(u_\varepsilon)=\pi \min\limits_{\overline{\Omega} }q^2\sqrt{det(K_H)}\ln\frac{1}{\varepsilon}+O(1).
\end{equation}
Moreover, there exist $ R_1,R_2>0 $ such that
\begin{equation*}
R_1\varepsilon\leq \text{diam}(A_\varepsilon)\leq R_2\varepsilon.
\end{equation*}
\end{proposition}
\begin{proof}
By \eqref{3-17}, we get for any $ 0<\tau<1 $ and $ \varepsilon\leq \tau $,
\begin{equation*}
\frac{2\pi}{\ln16\left (1+\frac{2\text{dist}(A_\varepsilon^\tau, \partial \Omega)}{\text{diam}(A_\varepsilon^\tau)}\right )}\leq \frac{C}{\ln\frac{1}{\tau}}.
\end{equation*}
So there exist $ C_0, \alpha_0>0 $ such that  $ \frac{\text{dist}(A_\varepsilon^\tau, \partial \Omega)}{\text{diam}(A_\varepsilon^\tau)}\geq \frac{C_0}{\tau^{\alpha_0}}$, which implies that $ \text{diam}(A_\varepsilon^\tau)\leq C_1\tau^{\alpha_0} $ for some $ C_1>0. $ That is, for any $ x,y \in A_\varepsilon^\tau $, $ |x-y|\leq C_1\tau^{\alpha_0} $.

By $ q\in C^{2}(\Omega)\cap C^{1}(\overline{\Omega}), (K_H)_{ij}\in C^{\infty}(\Omega)$ for $i,j=1,2$,  $ \inf_{\Omega}q>0 $ and $(\mathcal{K}2)$, similarly to the proof of \eqref{5001} and \eqref{5002}, we can get that $ q^2K_{H} $ is  Dini-continuous uniformly in $ \Omega $, which means that,  there exists a non-negative function $ \gamma(s), $ such that $ \int_0^{s_0}\frac{\gamma(s)}{s}ds<+\infty $ for some $ s_0>0 $ and
\begin{equation}\label{5003}
q^2(x)(K_H(x)\zeta|\zeta)\leq (1+\gamma(|x-y|))q^2(y)(K_H(y)\zeta|\zeta),\ \ \ \ \forall\ x,y\in\Omega,\ \ \zeta\in\mathbb{R}^2.
\end{equation}
Thus by \eqref{3-16} and \eqref{5003}, we get for any $ 0<\tau<\sigma<1 $, $ \varepsilon\leq \tau $ and $ x_\varepsilon\in A_\varepsilon\subseteq A_\varepsilon^\sigma $,
\begin{equation*}
\begin{split}
&\frac{1}{\varepsilon^2}\int_{\Omega}\left (u_\varepsilon-q\ln\frac{1}{\varepsilon}\right )_+^{p}qdx=\ln\frac{\sigma}{\tau}\int_{\Omega}q^2(K_H(x)\nabla w^{\sigma, \tau}_\varepsilon|\nabla w^{\sigma, \tau}_\varepsilon)dx\\
\geq &\ln\frac{\sigma}{\tau}\frac{1}{1+\gamma(C_1\sigma^{\alpha_0})}\int_{\Omega}q^2(x_\varepsilon)(K_H(x_\varepsilon)\nabla w^{\sigma, \tau}_\varepsilon|\nabla w^{\sigma, \tau}_\varepsilon)dx,
\end{split}
\end{equation*}
which implies that
\begin{equation}\label{3-24}
\left (\ln\frac{\sigma}{\tau}\right )^2\int_{\Omega}q^2(x_\varepsilon)(K_H(x_\varepsilon)\nabla w^{\sigma, \tau}_\varepsilon|\nabla w^{\sigma, \tau}_\varepsilon)dx\leq \ln\frac{\sigma}{\tau}(1+\gamma(C_1\sigma^{\alpha_0}))\frac{1}{\varepsilon^2}\int_{\Omega}\left (u_\varepsilon-q\ln\frac{1}{\varepsilon}\right )_+^{p}qdx.
\end{equation}
Taking $ \varepsilon=\tau_1<\sigma_1=\tau_2<\sigma_2=\tau_3<\cdots<\sigma_k=\rho $, and summing \eqref{3-24} over $ \{j=1,2,\cdots,k\} $, we get for any $ x_\varepsilon\in A_\varepsilon $
\begin{equation*}
\left (\ln\frac{\rho}{\varepsilon}\right )^2\int_{\Omega}q^2(x_\varepsilon)(K_H(x_\varepsilon)\nabla w^{\rho, \varepsilon}_\varepsilon|\nabla w^{\rho, \varepsilon}_\varepsilon)dx\leq \left (\ln\frac{\rho}{\varepsilon}+\sum_{j=1}^{k}\gamma(C_1\sigma_j^{\alpha_0})\ln\frac{\sigma_j}{\tau_j}\right )\frac{1}{\varepsilon^2}\int_{\Omega}\left (u_\varepsilon-q\ln\frac{1}{\varepsilon}\right )_+^{p}qdx.
\end{equation*}
By taking the limit of Riemann sums in the above inequality, we have
\begin{equation*}
\begin{split}
\left (\ln\frac{\rho}{\varepsilon}\right )^2\int_{\Omega}q^2(x_\varepsilon)(K_H(x_\varepsilon)\nabla w^{\rho, \varepsilon}_\varepsilon|\nabla w^{\rho, \varepsilon}_\varepsilon)dx\leq\left (\ln\frac{\rho}{\varepsilon}+\int_\varepsilon^\rho\frac{\gamma(C_1\sigma^{\alpha_0})}{\sigma}d\sigma\right )\frac{1}{\varepsilon^2}\int_{\Omega}\left (u_\varepsilon-q\ln\frac{1}{\varepsilon}\right )_+^{p}qdx.
\end{split}
\end{equation*}
Since
\begin{equation*}
\begin{split}
\int_\varepsilon^\rho\frac{\gamma(C_1\sigma^{\alpha_0})}{\sigma}d\sigma=\int_{\varepsilon^{\alpha_0}}^{\rho^{\alpha_0}}\frac{\gamma(C_1\sigma')}{\sigma'^{\frac{1}{\alpha_0}}}\frac{1}{\alpha_0}\sigma'^{\frac{1}{\alpha_0}-1}d\sigma'=\frac{1}{\alpha_0}\int_{C_1\varepsilon^{\alpha_0}}^{C_1\rho^{\alpha_0}}\frac{\gamma(\sigma')}{\sigma'}d\sigma'<+\infty,
\end{split}
\end{equation*}
we get
\begin{equation}\label{3-25}
\int_{\Omega}q^2(x_\varepsilon)(K_H(x_\varepsilon)\nabla w^{\rho, \varepsilon}_\varepsilon|\nabla w^{\rho, \varepsilon}_\varepsilon)dx\leq\left (\frac{1}{\ln\frac{\rho}{\varepsilon}}+\frac{C}{\left (\ln\frac{\rho}{\varepsilon}\right )^2}\right )\frac{1}{\varepsilon^2}\int_{\Omega}\left (u_\varepsilon-q\ln\frac{1}{\varepsilon}\right )_+^{p}qdx,
\end{equation}
which is  the refined version of \eqref{3-19}. So repeating the proof of Proposition \ref{loca}, we have
\begin{equation}\label{3-26}
q^2\sqrt{det(K_H)}(x_\varepsilon)\leq \frac{\ln16\left (1+\frac{C_0\text{dist}(A_\varepsilon, \partial \Omega)}{\text{diam}(A_\varepsilon)}\right )}{2\pi}\left (\frac{1}{\ln\frac{\rho}{\varepsilon}}+\frac{C}{\left (\ln\frac{\rho}{\varepsilon}\right )^2}\right )\left( \frac{1}{\varepsilon^2}\int_{\Omega}\left (u_\varepsilon-q\ln\frac{1}{\varepsilon}\right )_+^{p}qdx\right),
\end{equation}
which improves \eqref{3-21}.

Thus, taking \eqref{3-15} into \eqref{3-26} and using Proposition \ref{le301-upbdd}, we obtain
\begin{equation*}
\begin{split}
\min_{\overline{\Omega}}q^2\sqrt{det(K_H)}\ln\frac{1}{\varepsilon}\leq& \frac{\ln16\left (1+\frac{C_0\text{dist}(A_\varepsilon, \partial \Omega)}{\text{diam}(A_\varepsilon)}\right )}{2\pi}\left (\frac{1}{\ln\frac{\rho}{\varepsilon}}+\frac{C}{\left (\ln\frac{\rho}{\varepsilon}\right )^2}\right )(2c_\varepsilon+O(1))\\
\leq &\frac{\ln16\left (1+\frac{C_0\text{dist}(A_\varepsilon, \partial \Omega)}{\text{diam}(A_\varepsilon)}\right )}{2\pi}\left (\frac{1}{\ln\frac{\rho}{\varepsilon}}+\frac{C}{\left (\ln\frac{\rho}{\varepsilon}\right )^2}\right )\left (2\pi \min_{\overline{\Omega}}q^2\sqrt{det(K_H)}\ln\frac{1}{\varepsilon}+O(1)\right ).
\end{split}
\end{equation*}
Direct computation shows that
\begin{equation*}
\ln\frac{C_0\text{dist}(A_\varepsilon, \partial \Omega)}{\text{diam}(A_\varepsilon)}\geq \ln\frac{1}{\varepsilon}+O(1),
\end{equation*}
which implies that $ \text{diam}(A_\varepsilon)\leq R_2\varepsilon $ for some $ R_2>0. $

Finally, by taking $ \text{diam}(A_\varepsilon)\geq R_1\varepsilon $ (see Lemma \ref{lowbdd of diam}) into \eqref{3-26} and using \eqref{3-15}, we get
\begin{equation*}
\begin{split}
c_\varepsilon\geq& \pi\min_{\overline{\Omega}}q^2\sqrt{det(K_H)}\frac{1}{\ln16\left (1+\frac{C_0\text{dist}(A_\varepsilon, \partial \Omega)}{\text{diam}(A_\varepsilon)}\right ) }\frac{1}{\left (\frac{1}{\ln\frac{\rho}{\varepsilon}}+\frac{C}{\left (\ln\frac{\rho}{\varepsilon}\right )^2}\right )}\ln\frac{1}{\varepsilon}+O(1)\\
\geq&\pi\min_{\overline{\Omega}}q^2\sqrt{det(K_H)}\frac{1}{\ln\frac{C}{\varepsilon}}\frac{1}{\left (\frac{1}{\ln\frac{\rho}{\varepsilon}}+\frac{C}{\left (\ln\frac{\rho}{\varepsilon}\right )^2}\right )}\ln\frac{1}{\varepsilon}+O(1)\\
\geq& \pi\min_{\overline{\Omega}}q^2\sqrt{det(K_H)}\ln\frac{1}{\varepsilon}+O(1).
\end{split}
\end{equation*}
Combining this with Proposition \ref{le301-upbdd}, we get \eqref{3-27}. The proof is thus complete.

\end{proof}

\section{Proof of Theorem \ref{thm1} and \ref{thm01}}
\subsection{Proof of Theorem \ref{thm1}}
In  subsections 3.1 to 3.4, we prove the existence and asymptotic behavior of solutions   of \eqref{eq1} under the additional assumption that there exist minimum points of $ q^2\sqrt{det(K_H)} $ on the $ x_1 $-axis.

Now we give  proof of Theorem \ref{thm1} in the case that all minimum points of $ q^2\sqrt{det(K_H)} $ is not on the $ x_1 $-axis. Let $ \hat{x}=(|\hat{x}|\cos\theta_{\hat{x}}, |\hat{x}|\sin\theta_{\hat{x}}) $ be a minimizer of $ q^2\sqrt{det(K_H)} $ on $ \overline{\Omega} $.

Let $ \Omega_{\hat{x}}=\{\bar{R}_{\hat{x}}x\mid x\in \Omega\} $. For any function $ u\in \mathcal{H}(\Omega) $, let $ u_{\hat{x}}(x)=u(\bar{R}_{-\hat{x}}x)$ for any $  x\in\Omega_{\hat{x}}. $
So $ u_{\hat{x}}\in \mathcal{H}(\Omega_{\hat{x}}) $. Let $ q_{\hat{x}}(x)=q(\bar{R}_{-\hat{x}}x)$ for any  $x\in\Omega_{\hat{x}}. $ Then by Lemma \ref{le04}, we get that $ u $ is a solution of \eqref{eq1} if and only if
$ u_{\hat{x}} $ is a solution of
\begin{equation}\label{eq2}
\begin{cases}
-\text{div}(K_H(x)\nabla v)=\frac{1}{\varepsilon^2} \left (v-q_{\hat{x}}\ln\frac{1}{\varepsilon}\right )^{p}_+,\ \ \ \ &\text{in}\ \Omega_{\hat{x}}.\\
v=0,\ \ \ \ &\text{on}\ \partial \Omega_{\hat{x}}.
\end{cases}
\end{equation}

For equations \eqref{eq2}, we claim that there exist minimum points of $ q_{\hat{x}}^2\sqrt{det(K_H)} $ on the $ x_1 $-axis. Indeed, we can prove that $ (|\hat{x}|,0)=\bar{R}_{\hat{x}}\hat{x}\in \overline{\Omega_{\hat{x}}} $ and
\begin{equation*}
q_{\hat{x}}^2\sqrt{det(K_H)}(\bar{R}_{\hat{x}}\hat{x})=\frac{kq_{\hat{x}}^2(\bar{R}_{\hat{x}}\hat{x})}{\sqrt{k^2+|\bar{R}_{\hat{x}}\hat{x}|^2}}=\frac{kq^2(\hat{x})}{\sqrt{k^2+|\hat{x}|^2}}=\min_{\overline{\Omega}}q^2\sqrt{det(K_H)}=\min_{\overline{\Omega_{\hat{x}}}}q_{\hat{x}}^2\sqrt{det(K_H)},
\end{equation*}
which implies that $ (|\hat{x}|,0) $ is a minimizer  of $ q_{\hat{x}}^2\sqrt{det(K_H)} $.

Hence we can repeat the proof in subsections 2.1 and 3.1-3.4 to show that
there exist a family of solutions $ v_\varepsilon $ of \eqref{eq2} concentrating near minimum points of $ q_{\hat{x}}^2\sqrt{det(K_H)} $. Define $ u_\varepsilon(x)=v_\varepsilon(\bar{R}_{\hat{x}}x) $ for any $ x\in \Omega $. Then $\{ u_\varepsilon\} $ is a family of solutions of \eqref{eq1} which concentrates near a minimizer of $ q^2\sqrt{det(K_H)} $ as $ \varepsilon\to 0. $

Let $ w_\varepsilon=\mathcal{L}_H u_\varepsilon, \varphi_\varepsilon=u_\varepsilon -q\ln\frac{1}{\varepsilon}$. Then $ (w_\varepsilon, \varphi_\varepsilon) $ is the desired solution pair of Theorem \ref{thm1}.

\subsection{Proof of Theorem \ref{thm01}}
Based on results in Theorem \ref{thm1}, we now give the proof of Theorem \ref{thm01}. Let  $ q(x)=m $ for every $ m>0 $. Then  $ q $ satisfies $ \mathcal{L}_Hq=0 $. So by  Theorem \ref{thm1}, there exist  solutions $ u_\varepsilon $ of \eqref{rot eq2} with $ f(t)=t^p_+ $ and $ \mu=m\ln\frac{1}{\varepsilon} $ concentrating near $ x^* $, which is a minimizer of $ q^2\sqrt{det(K_H)} $. Since
\begin{equation*}
q^2\sqrt{det(K_H)}(x)=m^2\left( \frac{k^2}{k^2+|x|^2}\right) ^{\frac{1}{2}},
\end{equation*}
we get  that $ x^*\in\overline{\Omega} $ satisfies $ |x^*|=\max\limits_{\overline{\Omega}}|x| $. Let $ w_\varepsilon=\mathcal{L}_H u_\varepsilon  $. By Lemma \ref{circulation}, the limit of circulation is
\begin{equation*}
\lim\limits_{\varepsilon\to 0}\kappa(w_\varepsilon) =2\pi q(x^*)\sqrt{det(K_H(x^*))}=2\pi m\cdot \left( \frac{k^2}{k^2+|x^*|^2}\right) ^{\frac{1}{2}}=\frac{2 k\pi m}{\sqrt{k^2+|x^*|^2}}.
\end{equation*}
To conclude, $ (w_\varepsilon, u_\varepsilon) $ is the desired solution pair and the proof of Theorem \ref{thm01} is complete.

\begin{remark}
From the proof of Theorem \ref{thm01}, we see that the limiting location $ x^* $ of $ w_\varepsilon $ satisfies $ |x^*|=\max\limits_{\overline{\Omega}}|x| $. So $ x^* $ must be on the boundary of $ \Omega$.  This implies that results in subsection 3.4 can not hold. In this case,  the optimal estimates of diameter of $ A_\varepsilon $ is Lemma \ref{diam}, rather than Proposition \ref{better esti}.
\end{remark}

\subsection*{Acknowledgments:}

\par
D. Cao was supported by NNSF of China (grant No. 11831009) and Chinese Academy of Sciences by grant QYZDJ-SSW-SYS021. J. Wan was supported by NNSF of China (grant No. 12101045) and  Beijing
Institute of Technology Research Fund Program for Young Scholars (No.3170011182016).

\end{document}